%% file: paper.tex
\documentclass[11pt]{article}

\usepackage{amsthm,amsmath, amssymb, latexsym}
\usepackage{fullpage}
\usepackage[numbers]{natbib}
\usepackage{algorithm,algorithmic}
\usepackage{enumitem}
\usepackage{pdfsync}
\usepackage{graphicx}
\usepackage{color}
\usepackage{fancyhdr}
\usepackage{yfonts,mathrsfs,skak}
\usepackage[mathcal]{euscript}
\usepackage{rotating}
\usepackage[colorlinks=true,linkcolor=blue,citecolor=violet]{hyperref}

\let\inf\undef
\DeclareMathOperator*{\inf}{\vphantom{p}inf}
\let\sup\undef
\DeclareMathOperator*{\sup}{\vphantom{p}sup}
\let\max\undef
\DeclareMathOperator*{\max}{\vphantom{p}max}
\let\min\undef
\DeclareMathOperator*{\min}{\vphantom{p}min}

\usepackage{MnSymbol}

\input{macros}

\input{header}
\input{notation}

\newcommand{\multiminimax}[1]{\ensuremath{\left\llangle #1\right\rrangle}}

\usepackage{calc}

\newcommand{\bmu}{\ensuremath{\boldsymbol{\mu}}}

\newcommand{\Var}{{\bf Var}}

\title{\bf On Equivalence of Martingale Tail Bounds and Deterministic Regret Inequalities} 

\author{Alexander Rakhlin\\ University of Pennsylvania \and Karthik Sridharan\\ Cornell University}

\date{\today}

\makeindex

\begin{document}
\maketitle

\begin{abstract}
	We study an equivalence of (i) deterministic pathwise statements appearing in the online learning literature (termed \emph{regret bounds}), (ii) high-probability tail bounds for the supremum of a collection of martingales (of a specific form arising from uniform laws of large numbers for martingales), and (iii) in-expectation bounds for the supremum. By virtue of the equivalence, we prove exponential tail bounds for norms of Banach space valued martingales via deterministic regret bounds for the online mirror descent algorithm with an adaptive step size. We extend these results beyond the linear structure of the Banach space:  we define a notion of \emph{martingale type} for general classes of real-valued functions and show its equivalence (up to a logarithmic factor) to various sequential complexities of the class (in particular, the sequential Rademacher complexity and its offset version). For classes with the general martingale type 2, we exhibit a finer notion of variation that allows partial adaptation to the function indexing the martingale. Our proof technique rests on sequential symmetrization and on certifying the \emph{existence} of regret minimization strategies for certain online prediction problems.

\end{abstract}

\section{Introduction}

Let $Z_1,\ldots,Z_n$ be a martingale difference sequence taking values in a separable $(2,D)$-smooth Banach space $(\mathfrak{B},\|\cdot\|)$. A result due to Pinelis \cite{Pinelis94} asserts that for any $u>0$
\begin{align}
	\label{eq:pinelis}
	P\left( \sup_{n\geq 1} \norm{\sum_{t=1}^n Z_t}\geq \sigma u\right)\leq 2\exp\left\{-\frac{u^2}{2D^2}\right\},
\end{align}
where $\sigma$ is a constant satisfying $\sum_{t=1}^\infty \norm{Z_t}^2_\infty \leq \sigma^2$. Writing the norm $\norm{x} = \sup_{\norm{y}_*\leq 1} \inner{y,x}$ as the supremum over the dual ball, we may re-interpret \eqref{eq:pinelis} as a one-sided tail control for the supremum of a stochastic process $\left\{y\mapsto \sum_{t=1}^n \inner{y, Z_t}: \norm{y}_*\leq 1\right\}$. In this paper, we consider several extensions of \eqref{eq:pinelis}, motivated by the following questions:
\begin{enumerate}
	\setlength\itemsep{-1mm}
	\item[(a)] Can \eqref{eq:pinelis} be strengthened by replacing $\sigma$ with a ``path-dependent'' version of variation?
	\item[(b)] Does a version of \eqref{eq:pinelis} hold when we move away from the linear structure of the Banach space?
\end{enumerate}
Positive answers to these questions constitute the first contribution of our paper. The second contribution involves the actual technique. The cornerstone of our analysis is a certain equivalence of martingale inequalities and deterministic pathwise statements. The latter inequalities are studied in the field of online learning (or, sequential prediction), and are referred to as \emph{regret bounds}. We show that the existence (which can be certified via the minimax theorem) of prediction strategies that minimize regret yields predictable processes that help in answering (a) and (b). The equivalence is exploited in both directions, whereby stronger regret bounds are derived from the corresponding probabilistic bounds, and vice versa. To obtain one of the main results in the paper, we sharpen the bound by passing several times between the deterministic statements and probabilistic tail bounds. The equivalence asserts a strong connection between probabilistic inequalities for martingales and online learning algorithms.

In the remainder of this section, we present a simple example of the equivalence based on the gradient descent method, arguably the most popular convex optimization procedure. The example captures, loosely speaking, a correspondence between deterministic optimization methods and probabilistic bounds. Consider the unit Euclidean ball $\cB$ in $\reals^d$. Let $z_1,\ldots,z_n \in \cB$ and define, recursively, the Euclidean projections  
\begin{align}
	\label{eq:gradient_descent}
	\pred_{t+1} = \pred_{t+1}(z_1,\ldots,z_{t})=\text{Proj}_{\cB} \left(\pred_{t}-n^{-1/2}z_t\right)
\end{align}
for each $t=1,\ldots,n$, with the initial value $\pred_1=0$. Elementary algebra\footnote{See the two-line proof in the Appendix, Lemma~\ref{lem:gd}.} shows that for any $f\in\cB$, the regret inequality $\sum_{t=1}^n \inner{\pred_t-f,z_t} \leq \sqrt{n}$
holds deterministically for any sequence $z_1,\ldots,z_n\in \cB$. We re-write this statement as
\begin{align}
	\label{eq:l2regret_linear_norm}
	\norm{\sum_{t=1}^n z_t}-\sqrt{n} \leq \sum_{t=1}^n \inner{\pred_t,-z_t}.
\end{align}
Applying the deterministic inequality to a $\cB$-valued martingale difference sequence $-Z_1,\ldots,-Z_n$, 
\begin{align}
	\label{eq:tail_l2_regret}
	P\left(\norm{\sum_{t=1}^n Z_t}-\sqrt{n} > u\right) \leq P\left(\sum_{t=1}^n \inner{\pred_t,Z_t} > u\right) \leq \exp\left\{-\frac{u^2}{2n}\right\}.
\end{align}
The latter upper bound is an application of the Azuma-Hoeffding's inequality. Indeed, the process $(\pred_t)$ is predictable with respect to $\sigma(Z_1,\ldots,Z_t)$, and thus $(\inner{\pred_t,Z_t})$ is a $[-1,1]$-valued martingale difference sequence. It is worth emphasizing the conclusion: \emph{one-sided deviation tail bounds for a norm of a vector-valued martingale can be deduced from tail bounds for real-valued martingales with the help of a deterministic inequality}. Next, integrating the tail bound in \eqref{eq:tail_l2_regret} yields a seemingly weaker in-expectation statement
\begin{align}
	\label{eq:expected_l2_bound}
	\En\norm{\sum_{t=1}^n Z_t} \leq c\sqrt{n}
\end{align}
for an appropriate constant $c$. The twist in this uncomplicated story comes next: with the help of the minimax theorem,  \cite{RakSriTew10} established \emph{existence} of strategies $(\pred_t)$ such that 
\begin{align}
	\forall z_1,\ldots,z_n, f \in \mathcal{B},~~~~~ \sum_{t=1}^n \inner{\pred_t-f,z_t} \leq \sup\En\norm{\sum_{t=1}^n Z_t}, 
\end{align}
with the supremum taken over all martingale difference sequences with respect to a dyadic filtration. In view of \eqref{eq:expected_l2_bound}, this bound is $c\sqrt{n}$. 

What have we achieved? Let us summarize. The deterministic inequality \eqref{eq:l2regret_linear_norm}, which holds for \emph{all sequences}, implies a tail bound \eqref{eq:tail_l2_regret}. The latter, in turn, implies an in-expectation bound \eqref{eq:expected_l2_bound}, which implies \eqref{eq:l2regret_linear_norm} (with a worse constant) through a minimax argument, thus closing the loop. The equivalence---studied in depth in this paper---is informally stated below:
\paragraph{Informal:}
	The following bounds imply each other: (a) an inequality that holds for all sequences; (b) a deviation tail probability for the size of a  martingale; (c) an in-expectation bound on the size of a martingale.

The equivalence, in particular, allows us to amplify the in-expectation bounds to appropriate high-probability tail bounds. 

As already mentioned, the pathwise inequalities, such as \eqref{eq:l2regret_linear_norm}, are extensively studied in the field of online learning. In this paper, we employ some of the recently developed data-dependent (adaptive) regret inequalities to prove tail bounds for martingales. In turn, in view of the above equivalence, martingale inequalities shall give rise to novel deterministic regret bounds. 

While writing the paper, we learned of the \emph{trajectorial approach}, extensively studied in recent years. In particular, it has been shown that Doob's maximal inequalities and  Burkholder-Davis-Gundy inequalities have deterministic counterparts \cite{acciaio2013,beiglbock2014martingale,gushchin2014pathwise,beiglboeck2015pathwise}. The online learning literature contains a trove of pathwise inequalities, and further synthesis with the trajectorial approach (and the applications in mathematical finance) appears to be a promising direction.

This paper is organized as follows. In the next section, we extend the Euclidean result to martingales with values in Banach spaces, and also improve it by replacing $\sqrt{n}$ with square root of variation. We define a notion of martingale type for general classes of functions in Section~\ref{sec:gmtype}, and exhibit a tight connection to the growth of sequential Rademacher complexity. Section~\ref{sec:dyadic} presents sequential symmetrization; here we prove that statements for the dyadic filtration automatically yield corresponding tail bounds for general discrete-time stochastic processes. In Section~\ref{sec:regret}, we introduce the machinery for obtaining regret inequalities, and show how these inequalities  allow one to amplify certain in-expectation bounds into high-probability statements (Section~\ref{sec:amplification}). The last two sections contain some of the main results:  In Section \ref{sec:uniform} we prove a high probability bound for the notion of martingale type, and present a finer analysis of adaptivity of the variation term in Section \ref{sec:per_func}.

\section{Results in Banach spaces}
\label{sec:banach}

For the case of the Euclidean (or Hilbertian) norm, it is easy to see that the $\sqrt{n}$ bound of \eqref{eq:expected_l2_bound} can be improved to a distribution-dependent quantity $\left( \sum_{t=1}^n \En\norm{Z_t}^2 \right)^{1/2}$. Given the equivalence sketched earlier, one may wonder whether this implies existence of a gradient-descent-like method with a sequence-dependent variation governing the rate of convergence of this optimization procedure. Below, we indeed present such a method for $2$-smooth Banach spaces.

Let $(\mathfrak{B}, \|\cdot\|)$ be a separable Banach space, and let $(\mathfrak{B}_*, \|\cdot\|_*)$ denote its dual. $(\mathfrak{B}, \|\cdot\|)$ is of \emph{martingale type $p$} (for $p\in[1,2])$) if there exists a constant $C$ such that
\begin{align}
	\label{eq:ptype}
	\En\norm{\sum_{t=1}^n Z_t }^p \leq C^p \sum_{t=1}^n \En\norm{Z_t}^p
\end{align}
for any $\mathfrak{B}$-valued martingale difference sequence. The best possible constant $C$ in this inequality (as well as its finiteness) is known to depend on the geometry of the Banach space. For instance, for a Hilbert space \eqref{eq:ptype} holds for $p=2$ with constant $C=1$. On the other hand, triangle inequality implies that any space has the trivial type $p=1$.

An equivalent way to define martingale type $p$ is to ask that there exist a constant $C$ such that
\begin{align}
	\label{eq:2type-b}
	\E\sup_{\norm{y}_*\leq 1}~ \sum_{t=1}^n \inner{y,Z_t} = \En\norm{\sum_{t=1}^n Z_t} \leq C \left( \sum_{t=1}^n \En\norm{Z_t}^p \right)^{1/p}.
\end{align}

We now show that the strengthening to a sequence-dependent variation holds for any $2$-smooth Banach space, as we show next. Based on the equivalence mentioned earlier, we immediately obtain tail bounds.

Assume $\|\cdot\|$ is $2$-smooth. Let $D_{\cR}:\mathfrak{B}_*\times\mathfrak{B}_*\to \reals$ be the Bregman divergence with respect to a convex function $\cR$, which is assumed to be  $1$-strongly convex on the unit ball $\cB_*$ of $\mathfrak{B}_*$. Denote $R_{\max}^2 \deq \sup_{f,g\in\cB_*} D_{\cR}(f,g)$.  We extend and improve \eqref{eq:tail_l2_regret} as follows.
\begin{theorem}
	\label{thm:banach_space}
	Let $Z_1,\ldots,Z_n$ be a $\mathfrak{B}$-valued martingale difference sequence, and let $\En_{t}$ stand for the conditional expectation given $Z_1,\ldots,Z_{t}$. For any $u>0$, it holds that
	\begin{align}
		\label{eq:banach_tail}
		P\left( \frac{\norm{ \sum_{t=1}^n  Z_t} - 2.5 R_{\max} \left(\sqrt{V_n} + 1\right) }{ \sqrt{V_n + W_n  + \left(\En\sqrt{V_n+W_n}\right)^2}} > u  \right)\leq \sqrt{2}\exp\left\{-u^2/16\right\},
	\end{align}
	where
	\begin{align}
		\label{def:Vn_W_n}
		V_n = \sum_{t=1}^n  \norm{Z_t}^2 ~~~~~\text{and}~~~~ W_n = \sum_{t=1}^n \En_{t-1}\norm{Z_t}^2.
	\end{align}
	Furthermore, the bound holds with $W_n\equiv 0$ if the martingale differences are conditionally symmetric. 
\end{theorem}
In addition to extending the Euclidean result of the previous section to Banach spaces, \eqref{eq:banach_tail} offers several advantages. First, it is $n$-independent. Second, deviations are \emph{self-normalized} (that is, scaled by root-variation terms). We refer to Lemma~\ref{lem:azuma_hoeff_sup} for other forms of probabilistic bounds.

To prove the theorem, we start with a deterministic inequality from \citep[Corollary 2]{RakSri13nips}. For completeness, the proof is provided in the Appendix.
\begin{lemma}
	\label{lem:pred_seq_regret}
	Let $\F\subset\mathfrak{B}_*$ be a convex set. Define, recursively,
	\begin{align}
		\pred_{t+1} = \pred_{t+1}(z_1,\ldots,z_{t}) = \argmin{f\in\F}~ \eta_{t} \inner{f,z_t} + D_{\cR}(f,\pred_{t})
	\end{align}
	with $\pred_0=0$,  
	$\textstyle \eta_t \deq R_{\max}\min\left\{1, \left(\sqrt{\sum_{s=1}^{t} \norm{z_s}^2}+\sqrt{\sum_{s=1}^{t-1} \norm{z_s}^2}\right)^{-1}\right\}$, and 
	with $R_{\max}^2 \deq \sup_{f,g\in\F} D_{\cR}(f,g)$.
	Then for any $f\in \F$ and any $z_1,\ldots,z_n\in \mathfrak{B}$, 
	\begin{align*}
		\textstyle\sum_{t=1}^n \ip{\pred_t-f}{z_t} \le 2.5 R_{\max} \left(\sqrt{\sum_{t=1}^n \norm{z_t}^2 } + 1\right).
	\end{align*}
\end{lemma}

\begin{proof}[\textbf{Proof of Theorem~\ref{thm:banach_space}}]
We take $\F$ to be the unit ball in $\mathfrak{B}$, ensuring $\|\pred_t\|_* \leq 1$. For any martingale difference sequence $(Z_t)$ with values in $\mathfrak{B}_*$, the above lemma implies, by definition of the norm,  
\begin{align}
	\label{eq:banach_space_regret_norm}
	 \textstyle \norm{ \sum_{t=1}^n  Z_t} - 2.5 R_{\max} \left(\sqrt{V_n} + 1\right) \le \sum_{t=1}^n   \inner{\pred_t, Z_t}
\end{align}
for all sample paths. Dividing both sides by $\sqrt{V_n + W_n  + \left(\En\sqrt{V_n+W_n}\right)^2}$, we conclude that the left-hand side in \eqref{eq:banach_tail} is upper bounded by
\begin{align}
	\label{eq:tail_control_Banach_ratio}
	 P\left( \frac{\sum_{t=1}^n   \inner{\pred_t, Z_t}}{\sqrt{V_n + W_n  + \left(\En\sqrt{V_n+W_n}\right)^2}}  > u \right). 
\end{align}
To control this probability, we recall the following result  \cite[Theorem 2.7]{delaPenKlaLai07}:
\begin{theorem}[\cite{delaPenKlaLai07}]
	\label{thm:delapena_ratio}
	For a pair of random variables $A,B$, with $B>0$, such that
	\begin{align}
		\label{eq:moment_assumption_delapena}
		\E\exp\left\{\lambda A - \lambda^2 B^2/2\right\}\leq 1 ~~~~~ \forall \lambda\in\reals,
	\end{align}
	it holds that
	$$P\left(\frac{|A|}{\sqrt{B^2+(\En B)^2}} > u\right) \leq \sqrt{2}\exp\left\{-u^2/4\right\}.$$	
\end{theorem}

To apply this theorem, we verify assumption \eqref{eq:moment_assumption_delapena}:
\begin{lemma}
	\label{eq:AB-condition-delapena}
	The random variables $A = \sum_{t=1}^n \inner{\pred_t, Z_t}$ and $B^2 = 4\sum_{t=1}^n (\norm{Z_t}^2 + \En_{t-1}\norm{Z_t}^2)$ satisfy \eqref{eq:moment_assumption_delapena}.
\end{lemma}
The proof of the Lemma, as well as most of the proofs in this paper, is postponed to the Appendix. This concludes the proof of Theorem~\ref{thm:banach_space}.
\end{proof}

Let us make several remarks. First, \citep[Corollary 2]{RakSri13nips} proves a more general deterministic inequality: for any collection of functions $M_t = M_t(z_1,\ldots,z_{t-1})$, there exists a strategy $(\pred_t)$ such that
	\begin{align*}
		 \textstyle \forall z_1,\ldots,z_n \in \mathfrak{B},~~~~~~ \sum_{t=1}^n \ip{\pred_t-f}{z_t} \le 4.5 R_{\max} \left(\sqrt{\sum_{t=1}^n \norm{z_t-M_t}^2 } + 1\right).
	\end{align*}

	Second, the reader will notice that the pathwise inequality \eqref{eq:banach_space_regret_norm} does not depend on $n$ and the construction of $\pred_t$ is also oblivious to this value. A simple argument (Lemma~\ref{lem:BDG} in the Appendix) then allows us to lift the real-valued Burkholder-Davis-Gundy inequality (with the constant from \cite{burkholder2002best}) to the Banach space valued martingales:
	$$\En \max_{s=1,\ldots,n}  \norm{ \sum_{t=1}^s  Z_t}  \leq  \left(2.5 R_{\max}+\sqrt{3}\right) \En\sqrt{V_n} + 2.5 R_{\max}\ .
	$$
	Notably, the constant in the resulting BDG inequality is proportional to $R_{\max}$.

We also remark that Theorem~\ref{thm:banach_space} can be naturally extended to $p$-smooth Banach spaces $\mathfrak{B}$. This is accomplished in a straightforward manner  by extending Lemma~\ref{lem:pred_seq_regret}.

In conclusion, we were able to replace the distribution-independent $\sqrt{n}$ bound with a sequence-dependent quantity $V_n$. One may ask whether this phenomenon is general; that is, whether a sequence-dependent variation bound necessarily holds whenever the corresponding distribution-independent bound does. We prove in Theorem~\ref{thm:equivalence_type} below that this is indeed the case (up to a logarithmic factor), a result that holds for general classes of functions.

\section{Martingale Type for a General Class of Functions}\label{sec:gmtype}

We now define the analogue of a martingale type for a class $\G$ of real-valued measurable functions on some abstract measurable space $\cZ$. To this end, we assume that $(Z_1,\ldots,Z_n)$ is a discrete time process on a probability space $(\Omega, \mathcal{A}, P)$. Let $\E$ denote the expectation on this probability space, and let $\En_{t-1}$ denote the conditional  (given $Z_1,\ldots,Z_{t-1}$) expectation with respect to $Z_t$. For any $g:\cZ\to\reals$, 
\begin{align}
	\label{eq:gen_martingale_indexed_by_g}
	\sum_{t=1}^n \left( g(Z_t) - \Es{t-1}{g(Z_t)}\right)
\end{align}
is a sum of martingale differences $g(Z_t) - \Es{t-1}{g(Z_t)}$. We let $Z'_1,\ldots,Z'_n$ be a \emph{tangent sequence}; that is, $Z'_t$ and $Z_t$ are independent and identically distributed conditionally on $Z_{1},\ldots,Z_{t-1}$. Let $\En'_{t-1}$ denote the conditional (given $Z_1,\ldots,Z_{t-1}$) expectation with respect to $Z_{t}'$.

\begin{definition}
	\label{def:general_m_type}
 	A class $\G \subset \reals^\cZ$ has martingale type $p$ if there exists a  constant $C$ such that 
 \begin{align}
	 \label{eq:def_general_m_type}
	 \E{\left[\sup_{g \in \G} \sum_{t=1}^n \left(g(Z_t) - \Es{t-1}{g(Z_t)}\right)\right]} \le C\ \E{\left( \sum_{t=1}^n \E'_{t-1}\sup_{g \in \G}  \left|g(Z_t) - g(Z'_t)\right|^p \right)^{1/p}}.
 \end{align}
\end{definition}

\begin{remark}
	We conjecture that the statements below also hold for the definition of martingale type where $\E'_{t-1}\sup_{g \in \G}  \left|g(Z_t) - g(Z'_t)\right|^p$ on the right-hand side of \eqref{eq:def_general_m_type} is replaced with a smaller and more natural quantity $\sup_{g \in \G}  \left|g(Z_t) - \E'_{t-1} g(Z'_t)\right|^p$.
\end{remark}

In proving \eqref{eq:def_general_m_type}, we shall work with a dyadic filtration. Let $(\cA_t=\sigma(\epsilon_1,\ldots,\epsilon_t))_{t=1}^n$ generated by independent Rademacher (symmetric $\{\pm1\}$-valued) random variables $\epsilon_1,\ldots,\epsilon_n$. Let $\x = (\x_1,\ldots,\x_n)$ be a predictable process with respect to this filtration (that is, $\x_t$ is $\cA_{t-1}$-measurable) with values in some set $\X$. Sequential Rademacher complexity\footnote{This complexity is defined in \cite{RakSriTew14} without the absolute values; this difference is minor (and disappears if $0\in\F$).} of an abstract class $\F\subseteq\reals^\X$ on $\x$ is defined as 
\begin{align}
	\label{eq:def_seq_rad}
	\Rad_n(\F; \x) = \En \left| \sup_{f\in\F} \sum_{t=1}^n \epsilon_t f(\x_t)  \right| ~. 
\end{align}
\begin{definition}
	\label{def:growth_rate}
	Let $r\in (1,2]$. We say that sequential Rademacher complexity of $\F$ exhibits an $n^{1/r}$ growth with constant $C$ if 
\begin{align}
	\label{eq:def_r_growth}
	\forall n\geq 1,~~ \forall \x, ~~~ \Rad_n(\F; \x) \leq Cn^{1/r} \cdot \sup_{f\in\F,\epsilon\in\{\pm1\}^n,t\leq n} |f(\x_t(\epsilon))| ~.
\end{align}
\end{definition}

We will work with a particular class of functions $\F=\{f_g(z,z')=g(z)-g(z'): g\in\G\}$
defined on $\X\deq \cZ\times\cZ$. It is immediate that $\F$ exhibits $n^{1/r}$ whenever $\G$ does, and vice versa, with at most doubling of the constant $C$. 

Using a sequential symmetrization technique, it holds (see \cite{RakSriTew14}) that 
\begin{align}
	\E{\left[\sup_{g \in \G} \sum_{t=1}^n \left(g(Z_t) - \Es{t-1}{g(Z_t)}\right)\right]} \leq 2\sup_{\z} \Rad_n(\G; \z)~.
\end{align} 
Therefore, the statement ``$\G$ has martingale type $r$ whenever $\G$ exhibits an $n^{1/r}$ growth'' corresponds to the phenomenon that, loosely speaking, ``one may replace the distribution-independent $n^{1/r}$ bound with a sequence-dependent variation.''

The next theorem shows a tight connection between the complexity growth $n^{1/r}$ and martingale type.

\begin{theorem}
	\label{thm:equivalence_type}
	For any function class $\G \subseteq \reals^\cZ$, the following statements hold:
\begin{enumerate}
 \item	If for some $r\in (1,2]$ sequential Rademacher complexity exhibits $n^{1/r}$ growth, then $\G$ has martingale type $p$ for every $p <r$. 
 \item If $\G$ has martingale type $p$, then sequential complexity exhibits an $n^{1/p}$ growth.
 \end{enumerate}
\end{theorem}

The proof relies on the development in the next few sections, and especially on Lemma~\ref{lem:p_moment_lemma}. The technique is partly inspired by the work of Burkholder \cite{burkholder1964maximal} and Pisier \cite{Pisier75}. In particular, a key tool is the reverse H\"older principle  \cite[Prop. 8.53]{Pisier11}.

In addition to Theorem~\ref{thm:equivalence_type}, let us state informal versions of Theorems~\ref{thm:unif} and \ref{thm:perfunction} which appear, respectively, in Sections~\ref{sec:uniform} and \ref{sec:per_func}. Define the random variables
$$\Var_p = \En' \sum_{t=1}^n \sup_{g \in \G}  \left|g(Z_t) - g(Z'_t)\right|^p,~~~ \Var_p(g) = \En' \sum_{t=1}^n \left|g(Z_t) - g(Z'_t)\right|^p$$
where $\E'$ is expectation with respect to the tangent sequence, conditionally on $Z_{1:n}$. Then Theorem~\ref{thm:unif} states that with high probability controlled by $u>0$,
$$\sup_{g \in \G} \sum_{t=1}^n \left(g(Z_t) - \Es{t-1}{g(Z_t)}\right) ~\lesssim~ \log(n) \Var_r^{1/r} + u\Var_2^{1/2}$$
whenever $\G$ exhibits $n^{1/r}$ growth of sequential Rademacher complexity. Theorem~\ref{sec:per_func} addresses the case of martingale type $2$ and states that with high probability controlled by $u>0$,
$$\sup_{g \in \G} \sum_{t=1}^n \left(g(Z_t) - \Es{t-1}{g(Z_t)}\right) -  n^\frac{q}{4}(\Var_2^{1/2}(g))^{\frac{2-q}{4}} - u\Var^{1/2}_2(g) ~\lesssim~ 0$$
whenever sequential entropy (defined below) at scale $\alpha$ behaves as $\alpha^{-q}$. 

\subsection{Other complexity measures}

We see that the martingale type of $\G$ is described by the behavior of sequential Rademacher complexity. The latter behavior can, in turn, be quantified in terms of geometric quantities, such as sequential covering numbers and the sequential scale-sensitive dimension. We present the following two definitions from \cite{RakSriTew14}, both stated in terms of a predictable process $\x=(\x_1,\ldots,\x_n)$ with respect to the dyadic filtration.  It may be beneficial (at least it was for the authors of \cite{RakSriTew14}) to think of $\x$ as a complete binary tree of depth $n$, decorated by elements of $\cX$, and $\epsilon\in\{\pm1\}^n$ specifying a path in this tree.

\begin{definition}[Sequential covering number]
	\label{def:cover}
	Let $\x=(\x_1,\ldots,\x_n)$ be an $\X$-valued predictable process with respect to the dyadic filtration, and let $\F\subseteq\reals^\X$. A collection $V$ of $\reals$-valued predictable processes is called an $\alpha$-cover (with respect to $\ell_p$) of $\F$ on $\x$ if
	\begin{align}
		\forall f\in\F, ~~\forall \epsilon \in\{\pm1\}^n, ~~\exists \v\in V, ~~~~\mbox{s.t.}~~~~ \left(\frac{1}{n} \sum_{t=1}^n |f(\x_t(\epsilon))-\v_t(\epsilon)|^p\right)^{1/p} \leq \alpha ~.
	\end{align} 
	The cardinality of the smallest $\alpha$-cover is denoted by $\cN_p(\F,\alpha,\x)$ and $\cN_p(\F,\alpha,n) = \sup_{\x} \cN_p(\F,\alpha,\x)$, and both are referred to as  sequential covering numbers. Sequential entropy is defined as $\log \cN_p$.
\end{definition}

\begin{definition}[Sequential fat-shattering dimension]
	We say that $\F\subseteq \reals^\X$ shatters the predictable process $\x=(\x_1,\ldots,\x_n)$ at scale $\alpha>0$ if there exists a real-valued predictable process $\s$ such that 
	$$\forall \epsilon\in\{\pm1\}^n, ~~\exists f\in\F, ~~~~\mbox{s.t.}~~~~ \forall t\leq n,~~~~ \epsilon_t(f(\x_t(\epsilon))-\s_t(\epsilon)) \geq \alpha/2.$$
	The largest length $n$ of a shattered predictable process $\x$ is called the sequential fat-shattering dimension at scale $\alpha$ and denoted $\fat_\alpha(\F)$.
\end{definition}

The sequential covering numbers and the fat-shattering dimension are natural extensions of the classical notions, as shown in \cite{RakSriTew14}. In particular, a Dudley-type entropy integral upper bound in terms of sequential covering numbers holds for sequential Rademacher complexity. The sequential covering numbers, in turn, are upper bounded in terms of the fat-shattering dimensions, in a parallel to the way classical empirical covering numbers are controlled by the scale-sensitive version of the Vapnik-Chervonenkis dimension. We summarize the implications of these relationships in the following corollary:
\begin{corollary}
	For any function class $\F \subseteq \reals^\X$,
\begin{enumerate}
	\item If for some $q>0$ either $\forall \alpha, \log\mathcal{N}_2(\F,\alpha,n) \le C \alpha^{-q}$ or $\forall \alpha, \fat_\alpha(\F)\leq C\alpha^{-q}$, then $\F$ has martingale type $p$ for any $p < \frac{\max\{q,2\}}{\max\{q,2\}-1}$.
	\item If $\F$ has martingale type $r\in(1,2]$ then, for every $p<r$, there exists $C$ such that $\log\mathcal{N}_2(\F,\alpha,n) \le C \alpha^{-\frac{p}{p-1}}$ and $\fat_\alpha(\F)\leq C\alpha^{-\frac{p}{p-1}}$, for all $\alpha$.
\end{enumerate}
\end{corollary}

We have established a relation between the martingale type of a function class $\F$ and several sequential complexities of the class. However, unlike our starting point \eqref{eq:pinelis} and Theorem~\ref{thm:banach_space}, our results so far do not quantify the tail behavior for the difference between the supremum of the martingale process and the corresponding variation. A natural idea is to mimic the ``equivalence'' argument used in Section~\ref{sec:banach} to conclude the exponential tail bounds. Unfortunately, the deviation inequalities of the previous section rest on pathwise regret bounds that, in turn, rely on the linear structure of the associated Banach space, as well as on properties such as smoothness and uniform convexity. Without the linear structure, it is not clear whether the analogous pathwise statements hold. The goal of the rest of the paper is to bring forth some of the tools recently developed within the online learning literature, and to apply these pathwise regret bounds to conclude \emph{high probability tail bounds associated to martingale type}. In addition to this goal, we will seek a version of Theorem~\ref{thm:equivalence_type}(i) for bounded functions, where the $n^{1/r}$ growth of sequential Rademacher complexity implies martingale type $r$ (rather than any $p<r$), but with an additional $\log(n)$ factor. Our third goal will be to establish per-function variation bounds (similar to the notion of a weak variance \cite{boucheron2013concentration}). We show that this latter bound is a finer version of the variation term, possible for classes that are ``not too large''.

Our plan is as follows. First, we reduce the problem to one based on the dyadic filtration. After that, we shall introduce certain deterministic inequalities from the online learning literature that are already stated for the dyadic filtration.

\section{Symmetrization: dyadic filtration is enough}
\label{sec:dyadic}

The purpose of this section is to prove that statements for the dyadic filtration can be lifted to general processes via sequential symmetrization. Consider the martingale
$$M_g = \sum_{t=1}^n g(Z_t)-\En[g(Z_t)|Z_1,\ldots,Z_{t-1}]$$
indexed by $g\in\G$. If $(Z_t)$ is adapted to a dyadic filtration $\cA_{t} = \sigma(\epsilon_1,\ldots,\epsilon_t)$, each increment
$g(Z_t)-\En[g(Z_t)|Z_1,\ldots,Z_{t-1}]$ takes on the value 
$$f_g(\x_t(\epsilon_{1:t-1})) \deq \left( g(Z_t(\epsilon_{1:t-1},+1))-g(Z_t(\epsilon_{1:t-1},-1))\right)/2$$
or its negation, where $\x_t$ is a predictable process with values in $\cZ\times \cZ$ and $f_g\in \F$ defined by $(z,z')\mapsto g(z)-g(z')$. In the rest of the paper, we work directly with martingales of the form $M_f = \sum_{t=1}^n \epsilon_t f(\x_t(\epsilon))$, indexed by an abstract class $\F\subseteq \reals^\X$ and an abstract $\X$-valued predictable process $\x$.

We extend the symmetrization approach of Panchenko \cite{Pan03sym} to sequential symmetrization for the case of martingales. In contrast to the more frequently-used Gin\'e-Zinn symmetrization proof (via Chebyshev's inequality) \cite{GinZin84,VanDerVaartWe96} that allows a direct tail comparison of the symmetrized and the original processes, Panchenko's approach allows for an ``indirect'' comparison. The following immediate extension of \cite[Lemma 1]{Pan03sym} will imply that any $\exp\{-\mu(u)\}$ type tail behavior of the symmetrized process yields the same behavior for the original process.

\begin{lemma}
	\label{lem:tail_comparison}
	Suppose $\xi$ and $\nu$ are random variables and for some $\Gamma\geq 1$ and for all $u\geq 0$
	$$P(\nu\geq u)\leq \Gamma \exp\{- \mu(u)\}.$$
	Let $\mu:\reals_+ \to\reals_+$ be an increasing differentiable function with $\mu(0)=0$ and $\mu(\infty)=\infty$. Suppose for all $a\in\reals$ and $\phi(x)\deq \mu([x-a]_+)$ it holds that
	$\En\phi(\xi)\leq \En\phi(\nu).$
	Then for any $u\geq 0$,
	$$P(\xi\geq u)\leq \Gamma \exp\{- \mu(u-\mu^{-1}(1))\}.$$
	In particular, if $\mu(b)=cb$, we have 
	$P(\xi\geq u)\leq \Gamma \exp\{1- cu\}$; 
	if $\mu(b)=cb^2$, then 
	$P(\xi\geq u) \leq \Gamma \exp\{1- cu^2/4\}.$
\end{lemma}

As in \cite{Pan03sym}, the lemma will be used with $\xi$ and $\nu$ as functions of a single sample and the double sample, respectively. The expression for the double sample will be symmetrized in order to pass to the dyadic filtration. However, unlike \cite{Pan03sym}, we are dealing with a dependent sequence $Z_1,\ldots,Z_n$, and the meaning ascribed to the ``second sample'' $Z_1',\ldots,Z_n'$ is that of a \emph{tangent sequence}. That is, $Z_t,Z_t'$ are independent and have the same distribution conditionally on $Z_1,\ldots,Z_{t-1}$. Let $\En_{t-1}$ stand for the conditional expectation given $Z_1,\ldots,Z_{t-1}$. 

\begin{corollary}
	\label{cor:prob_symmetrization}
	Let $\tilde{B}:\G\times \cZ^{2n} \to \reals$ be a function that is symmetric with respect to the swap of the $i$-th pair $z_i,z_i'$, for any $i\in[n]$: 
\begin{align}
	\label{eq:symmetry_condition_B}
	\tilde{B}(g; z_1,z_1',\ldots,z_i,z_i',\ldots,z_n,z_n') = \tilde{B}(g; z_1,z_1',\ldots,z_i',z_i,\ldots,z_n,z_n')
\end{align}
for all $g\in\G$. Then, under the assumptions of Lemma~\ref{lem:tail_comparison} on $\mu$, a tail behavior
\begin{align*}
	\forall (\z,\z'),~~ P\left( \sup_{g\in\G} \sum_{t=1}^{n} \epsilon_t (g(\z_t)-g(\z'_t))- \tilde{B}(g; (\z_1,\z'_1),\ldots,(\z_n,\z'_n))>u \right) \leq \Gamma \exp\{- \mu(u)\}
\end{align*}
for all $u>0$ implies the tail bound
\begin{align*}
	&P\left(\sup_{g\in\G} \sum_{t=1}^n (g(Z_t)-\En_{t-1} g(Z_t)) - \En_{Z_{1:n}'} \tilde{B}(g; Z_1,Z_1',\ldots,Z_n, Z_n')>u\right) \leq \Gamma \exp\{- \mu(u-\mu^{-1}(1))\}
\end{align*}
for any sequence of random variables $Z_1,\ldots,Z_n$ and the corresponding tangent sequence $Z_1',\ldots,Z_n'$. The supremum is taken over a pair of predictable processes $\z,\z'$ with respect to the dyadic filtration. 
A direct comparison of the expected suprema also holds:
	\begin{align}
		&\En\sup_{g\in\G} \sum_{t=1}^n (g(Z_t)-\En_{t-1} g(Z_t)) - \En_{Z_{1:n}'} \tilde{B}(g; Z_1,Z_1',\ldots,Z_n, Z_n') \\
		&\hspace{2in}\leq \sup_{\z,\z'} \En \sup_{g\in\G} \sum_{t=1}^{n} \epsilon_t (g(\z_t)-g(\z'_t))- \tilde{B}(g; (\z_1,\z'_1),\ldots,(\z_n,\z'_n)). \notag
	\end{align}
\end{corollary}

We conclude that it is enough to prove tail bounds for a supremum 
$$\textstyle\sup_{f\in\F} \sum_{t=1}^n \epsilon_t f(\x_t) - B(f; \x_1,\ldots,\x_n)$$
of a martingale with respect to the dyadic filtration, offset by a function $B(f; \x_1,\ldots,\x_n)$. This will be achieved with the help of deterministic regret inequalities.

\section{Deterministic regret inequalities}
\label{sec:regret}

\subsection{Sequential prediction}
\label{sec:secpred}

We let $y_1,\ldots,y_n \in\{\pm1\}$ and $x_1,\ldots,x_n \in \X$ for some abstract measurable set $\X$. Let $\F$ be a class of $[-1,1]$-valued functions on $\X$. Fix a cost function $\loss:\reals\times\reals\to\reals$, convex in the first argument. For a given function $B:\F\times\X^n\to \reals$, we aim to construct $\pred_t = \pred_t(x_1,\ldots,x_t, y_1,\ldots,y_{t-1})\in [-1,1]$ such that
\begin{align}
	\label{eq:gen_regret}
	\forall~ (x_t,y_t)_{t=1}^n, ~~~ \sum_{t=1}^n \loss(\pred_t,y_t) \le \inf_{f \in \F}\left\{ \sum_{t=1}^n \loss(f(x_t),y_t) +  B(f;x_1,\ldots,x_n)\right\}.
\end{align}
We may view $\pred_t$ as a prediction of the next value $y_t$ having observed $x_t$ and all the data thus far. In this paper, we focus on the linear loss  $\loss(a,b)=-ab/2$ (equivalently, absolute loss $|a-b|=(1-ab)/2$ when $b\in\{\pm1\}$) and $\loss(a,b)=(a-b)^2$. 
We equivalently write \eqref{eq:gen_regret} for the linear cost function as
\begin{align}
	\label{eq:abs_regret}
	\sup_{f \in \F}\left\{ \sum_{t=1}^n y_tf(x_t) - 2B(f;x_1,\ldots,x_n)\right\} \leq \sum_{t=1}^n y_t \pred_t 
\end{align}
while for the square loss it becomes
\begin{align}
	\label{eq:sq_regret}
	\sup_{f \in \F}\left\{ \sum_{t=1}^n 2y_tf(x_t)-f(x_t)^2 - B(f;x_1,\ldots,x_n)\right\} \leq \sum_{t=1}^n 2y_t \pred_t -\pred_t^2.
\end{align}
Given a function $B$ and a class $\F$, there are two goals we may consider: (a) certify the existence of $(\pred_t)\deq (\pred_1,\ldots,\pred_n)$ satisfying the pathwise inequality \eqref{eq:gen_regret} for all sequences $(x_t,y_t)_{t=1}^n$; or (b) give an explicit construction of $(\pred_t)$. Both questions have been studied in the online learning literature, but the non-constructive approach will play an especially important role. Indeed, explicit constructions---such as the simple gradient descent update  \eqref{eq:gradient_descent} --- might not be available in more complex situations, yet it is the \emph{existence} of $(\pred_t)$ that yields the sought-after tail bounds.

\subsection{Existence of strategies}
To certify the existence of a strategy $(\pred_t)$, consider the following object:
\begin{align}
	\label{eq:minimax_def}
	\cA(\F,B) = \multiminimax{\sup_{x_t}\inf_{\hat{y}_t}\max_{y_t}}_{t=1}^n \left\{ \sum_{t=1}^n \loss(\hat{y}_t,y_t) - \inf_{f \in \F}\left\{ \sum_{t=1}^n \loss(f(x_t),y_t) +  B(f;x_1,\ldots,x_n)\right\} \right\}
\end{align}
where the notation $\multiminimax{\cdots}_{t=1}^n$ stands for the repeated application of the operators (the outer operators corresponding to $t=1$). The variable $x_t$ ranges over $\X$, $y_t$ is in the set $\{\pm1\}$, and $\hat{y}_t$ ranges in $[-1,1]$.
It follows that 
\begin{quote}
	$\cA(\F,B)\leq 0$ is a \emph{necessary and sufficient condition} for the existence of $(\pred_t)$ such that \eqref{eq:gen_regret} holds. 
\end{quote}
Indeed, the optimal choice for $\hat{y}_1$ is made given $x_1$; the optimal choice for $\hat{y}_2$ is made given $x_1,y_1,x_2$, and so on. This choice defines the optimal strategy $(\pred_t)$.\footnote{If the infima are not achieved, a limiting argument can be employed.} The other direction is immediate.

Suppose we can find an upper bound on $\cA(\F,B)$ and then prove that this upper bound is non-positive. This would serve as a \emph{sufficient} condition for the existence of $(\pred_t)$. Next, we present such an upper bound for the case when the cost function is linear. More general results for convex Lipschitz cost functions can be found in \cite{FosRakSri15}.

As before, let $\epsilon=(\epsilon_1,\ldots,\epsilon_n)$ be a sequence of independent Rademacher random variables. Let $\x=(\x_1,\ldots,\x_n)$ and $\y=(\y_1,\ldots,\y_n)$ be predictable processes with respect to the dyadic filtration $\sigma(\epsilon_1,\ldots,\epsilon_t)$, with values in $\X$ and $\{\pm1\}$, respectively. In other words, $\x_t=\x_t(\epsilon_1,\ldots,\epsilon_{t-1})\in \X$ and $\y_t=\y_t(\epsilon_1,\ldots,\epsilon_{t-1})\in\{\pm1\}$ for each $t=1,\ldots,n$.

\begin{lemma}
	\label{lem:sufficient_for_abs_regret}
	For the case of the linear cost function, 
	\begin{align}
		\label{eq:value_upper1}
		\cA(\F,B)\leq \sup_{\x} \En \left[ \sup_{f\in\F} \sum_{t=1}^n \frac{1}{2}\epsilon_t f(\x_t) - B(f; \x_1,\ldots,\x_n) \right].
	\end{align}
	Therefore, whenever it holds that for any predictable process $\x=(\x_1,\ldots,\x_n)$ 
	\begin{align}
		\label{eq:compare_expectations}
		\En \left[ \sup_{f\in\F} \sum_{t=1}^n \epsilon_t f(\x_t)- 2 B(f; \x_1,\ldots,\x_n) \right] \leq 0 ~, 
	\end{align}
	there exists a strategy $(\pred_t)$ with values
	\begin{align}
		\label{eq:pred_ranges}
		\textstyle |\pred_t| \leq \sup_{f\in\F}|f(x_t)| 
	\end{align}
	such that the pathwise inequality \eqref{eq:abs_regret} holds.
\end{lemma}

Condition \eqref{eq:compare_expectations} in the previous lemma implies the existence of a strategy for \eqref{eq:abs_regret}. However, there might be situations when \eqref{eq:compare_expectations} can be verified for a function $B(f;\x)$ of the predictable process that does not have a corresponding representation in the sense of \eqref{eq:abs_regret}. The next lemma provides a variant of Lemma~\ref{lem:sufficient_for_abs_regret}.
\begin{lemma}
	\label{lem:per_tree_abs_regret}
	Let $\x$ be an $\X$-valued predictable process with respect to the dyadic filtration. Let the function $B$ map the predictable process $\x$ and a function $f\in\F$ to a real value, with the property 
	\begin{align}
		\label{eq:rotation_decreases}
		\sup_{\y} B(f; \x\circ\y) \leq B(f; \x)
	\end{align}
	where $\y=(\y_1,\ldots,\y_n)$ is a $\{\pm1\}$-valued predictable process, and $(\x\circ \y)_t = \x_t(\y_{2:t}(\epsilon))$. If 
	\begin{align}
		\label{eq:per_tree_condition}
		\En  \left[ \sup_{f\in\F} \sum_{t=1}^n \epsilon_t f(\x_t) - 2 B(f; \x) \right] \leq 0,
	\end{align}
	then there is a strategy $(\pred_t)$ with $\pred_t = \pred_t(y_1,\ldots,y_{t-1})$ and $|\pred_t|\leq \sup_{f\in\F} |f(\x_t)|$ such that 
	\begin{align}
		\label{eq:abs_regret_per_tree}
		\forall y_1,\ldots,y_n\in \{\pm1\},~~~~ \sup_{f \in \F}\left\{ \sum_{t=1}^n y_tf(\x_t(y_1,\ldots,y_{t-1})) - 2B(f;\x)\right\} \leq \sum_{t=1}^n y_t \pred_t. 
	\end{align}
\end{lemma}

\section{Amplification and equivalence}
\label{sec:amplification}

We now describe an interesting amplification phenomenon, already presented in the Introduction for the simple Euclidean case. Whenever \eqref{eq:compare_expectations} holds, the deterministic inequality \eqref{eq:abs_regret} holds, and, therefore, we may apply it to a particular martingale difference sequence to obtain high-probability bounds. Below, we detail this amplification for both linear and square loss functions.

\subsection{Linear loss}

Take any $\X$-valued predictable process $\x=(\x_1,\ldots,\x_n)$ with respect to the dyadic filtration. The deterministic inequality \eqref{eq:abs_regret} applied to $x_t=\x_t(\epsilon_1,\ldots,\epsilon_{t-1})$ and $y_t=\epsilon_t$ becomes 
\begin{align}
	\label{eq:abs_regret_martingale}
	\sup_{f \in \F}\left\{ \sum_{t=1}^n \epsilon_t f(\x_t) - 2B(f;\x_1,\ldots,\x_n)\right\} \leq \sum_{t=1}^n \epsilon_t \pred_t 
\end{align}
for any $\epsilon$, and thus we have the comparison of tails
\begin{align}
	\label{eq:prob_tail_compare}
	P\left( \sup_{f \in \F}\left\{ \sum_{t=1}^n \epsilon_t f(\x_t) - 2B(f;\x_1,\ldots,\x_n)\right\} > u\right) \leq P\left(\sum_{t=1}^n \epsilon_t \pred_t > u\right).
\end{align}
Given the boundedness of the increments $\epsilon_t\pred_t$, the tail bounds follow immediately from the Azuma-Hoeffding's inequality or from Freedman's inequality \cite{freedman1975tail}. More precisely, we use the fact that the martingale differences are bounded by $|\pred_t|\leq \sup_{f\in\F} |f(\x_t)|$, and conclude
\begin{lemma}
	\label{lem:azuma_hoeff_sup}
	If there exists a prediction strategy $(\pred_t)$ that satisfies \eqref{eq:abs_regret} and \eqref{eq:pred_ranges}, then for any predictable process $\x$ Azuma-Hoeffding inequality implies that
	\begin{align}
		\label{eq:azuma_hoeff_sup}
	 	P\left(\sup_{f \in \F}\left\{ \sum_{t=1}^n \epsilon_t f(\x_t) - 2B(f;\x_1,\ldots,\x_n) \right\} > u \right) \le \exp\left( - \frac{u^2}{4 \max_{\epsilon} \sum_{t=1}^n\sup_{f \in \F} f(\x_t(\epsilon))^2  }\right),
	\end{align}
	Freedman's inequality implies 
	\begin{align}
		\label{eq:freedman_sup}
	 	P\left(\sup_{f \in \F}\left\{ \sum_{t=1}^n \epsilon_t f(\x_t) - 2B(f;\x_1,\ldots,\x_n) \right\} > u,~~  \sum_{t=1}^n\sup_{f \in \F} f(\x_t)^2  \leq \sigma^2 \right) \le \exp\left( - \frac{u^2}{2\sigma^2 + 2uM/3 }\right),
	\end{align}
	where $M=n\cdot \sup_{f\in\F,\epsilon\in\{\pm1\}^n,t\leq n}|f(\x_t)| $, and we also have that for any $\alpha>0$,
	\begin{align}
		\label{eq:exp_tail_sup}
		P\left(\sup_{f \in \F}\left\{ \sum_{t=1}^n \epsilon_t f(\x_t) - 2B(f;\x_1,\ldots,\x_n)\right\} - \alpha \sum_{t=1}^n\sup_{f \in \F} f(\x_t)^2  > u \right) \leq \exp\left(-2\alpha u\right).
	\end{align}
	 In view of Lemma~\ref{lem:sufficient_for_abs_regret}, a sufficient condition for these inequalities is that \eqref{eq:compare_expectations} holds for all $\x$. The same inequalities hold with $B(f; \x)$ if conditions of Lemma~\ref{lem:per_tree_abs_regret} are verified for the given $\x$.
\end{lemma}

Let us emphasize the conclusion of the above lemma: \emph{the non-positivity of the expected supremum of a collection of martingales, offset by a function $2B$, implies existence of a regret-minimization strategy, which implies a high-probability tail bound.} To close the loop, we integrate out the tails, obtaining an in-expectation bound of the form \eqref{eq:compare_expectations}, but possibly with a larger $B$ function. This is a more general form of the equivalence promised in the introduction.

The next goal is to find nontrivial functions $B$ such that \eqref{eq:compare_expectations} holds. The most basic $B$ is a constant that depends on the complexity of $\F$, but not on $f$ or the data. Define the worst-case sequential Rademacher averages as
	\begin{align}
		\label{eq:wc_seq_rad}
		\Rad_n(\F) ~\deq~ \sup_\x \En\sup_{f \in \F} \sum_{t=1}^n \epsilon_t f(\x_t).
	\end{align}
	Clearly, $B = \Rad_n(\F)/2$ satisfies \eqref{eq:compare_expectations}. The following is immediate.
\begin{corollary}
	\label{cor:dev_above_seq_rad}
	For any $\F\subseteq \reals^\X$ and an $\X$-valued predictable process $\x$ with respect to the dyadic filtration, 	
	\begin{align}
		\label{eq:dev_above_seq_rad}
		P\left(\sup_{f \in \F} \sum_{t=1}^n \epsilon_t f(\x_t) > \Rad_n(\F) + u \right) \leq \exp\left( - \frac{u^2}{4 \max_{\epsilon} \sum_{t=1}^n\sup_{f \in \F} f(\x_t(\epsilon))^2  }\right).
	\end{align}
\end{corollary}
Superficially, \eqref{eq:dev_above_seq_rad} looks like a one-sided version of the concentration bound for classical (i.i.d.) Rademacher averages \cite{boucheron2013concentration}. However, sequential Rademacher averages are not Lipschitz with respect to a flip of a sign, as the whole remaining path may change after a flip.

\subsection{Square loss}

As for the case of the linear loss function, take any $\X$-valued predictable process $\x=(\x_1,\ldots,\x_n)$ with respect to the dyadic filtration. Fix $\alpha>0$. The deterministic inequality \eqref{eq:sq_regret} for $x_t=\x_t(\epsilon_1,\ldots,\epsilon_{t-1})$ and $y_t=\frac{1}{\alpha}\epsilon_t$ becomes 
\begin{align}
	\label{eq:sq_subst_process}
	\sup_{f \in \F}\left\{  \sum_{t=1}^n \left( \frac{2}{\alpha} \epsilon_t f(\x_t) - f^2(\x_t) \right)  -  B(f;\x_1,\ldots,\x_n)\right\}  \le \sum_{t=1}^n \frac{2}{\alpha} \epsilon_t \pred_t - \pred_t^2 \ .
\end{align}
As in the proof of \eqref{eq:exp_tail_sup}, we obtain a tail comparison
\begin{align*}
P\left(\sup_{f \in \F}\left\{  \sum_{t=1}^n\left(  \frac{2}{\alpha} \epsilon_t f(\x_t) - f^2(\x_t)\right) -  B(f;\x_1,\ldots,\x_n)\right\} > \frac{u}{\alpha}  \right) \le P\left(\sum_{t=1}^n \left( \frac{2}{\alpha} \epsilon_t \pred_t  - \pred_t^2 \right) > \frac{u}{\alpha} \right) \le \exp\left\{ - \frac{\alpha u}{2}\right\} \ .
\end{align*}

Once again, the most basic choice for $B$ is the constant that depends on the complexity of the class. We recall the following result from \cite{RakSri14nonparam}.
\begin{lemma}[\cite{RakSri14nonparam}] 
 Let $\kappa>0$. For any class $\F\subseteq\reals^\X$, there exists a prediction strategy $(\pred_t)$ with values in $[-\kappa,\kappa]$ such that
$$
\forall (x_1,y_1),\ldots, (x_n,y_n) \in \X\times[-\kappa,\kappa], ~~~~~ \sum_{t=1}^n (\pred_t - y_t)^2 - \inf_{f \in \F} \sum_{t=1}^n (f(x_t) - y_t)^2 \le \Rad_n^{\mrm{off}}(\F, \kappa, 1),
$$
where, analogously to \eqref{eq:wc_seq_rad}, we define \emph{offset Rademacher complexity} 
\begin{align}
		\label{eq:off_seq_rad}
		\Rad_n^{\mrm{off}}(\F,c_1,c_2) \deq \sup_{\x,\bmu} \En\sup_{f \in \F} \left\{ \sum_{t=1}^n 4 c_1 \epsilon_t (f(\x_t)-\bmu_t) - c_2 (f(\x_t)-\bmu_t)^2 \right\} .
	\end{align}
Here, the supremum is taken over $\X$-valued predictable processes $\x=(\x_1,\ldots,\x_n)$ and $[-\kappa,\kappa]$-valued predictable processes $\bmu$, both with respect to the dyadic filtration. 
\end{lemma}
We conclude that \eqref{eq:sq_subst_process} is satisfied with the data-independent constant  $B=\Rad_n^{\mrm{off}}(\F, 1/\alpha,1)$. Hence, the following analogue of Corollary~\ref{cor:dev_above_seq_rad} holds:
\begin{corollary}
	\label{cor:dev_above_off_rad}
	Let $\F\subseteq[-1,1]^\X$. For any $\X$-valued predictable process $\x$ with respect to the dyadic filtration and for any $\alpha>0$, it holds that
	\label{cor:dev_above_offset_rad}
	$$P\Bigg(\sup_{f \in \F}\left\{  \sum_{t=1}^n\left(  2 \epsilon_t f(\x_t) - \alpha f^2(\x_t)\right) \right\} - \Rad_n^{\mrm{off}}(\F, 1, \alpha) >   u \Bigg)   \leq \exp\left\{-\frac{\alpha u}{2}\right\} ~.$$ 
\end{corollary}

To summarize, in Section~\ref{sec:regret} we presented the machinery of regret inequalities, as well as sufficient conditions for existence of strategies. In the present section we used the pathwise statements, along with real-valued deviation  inequalities, to conclude tail bounds, which, in turn, certify existence of regret-minimization strategies. In the next two sections we put these techniques to use.

\section{Uniform variation and tail bounds for general martingale type}
\label{sec:uniform}

\newcommand{\bb}{\boldsymbol{b}}

We now make an extensive use of the amplification technique to prove in-probability versions of the ``martingale type'' definition. We start by working with dyadic martingales of the form $f\mapsto \sum_{t=1}^n \epsilon_t f(\x_t)$ where $\x=(\x_1,\ldots,\x_n)$ is a predictable process (with respect to the dyadic filtration) with values in $\cX$. Once the results for these objects are established, we conclude the corresponding statements for general processes of the form \eqref{eq:gen_martingale_indexed_by_g} via the sequential symmetrization technique summarized in Corollary~\ref{cor:prob_symmetrization}.

As in Section~\ref{sec:gmtype}, we assume a growth condition $n^{1/r}$ on sequential Rademacher complexity. 

\begin{lemma}
	\label{lem:p_moment_lemma}
  Let $\F\subseteq \reals^\X$ and $r\in(1,2]$. Under the growth assumption \eqref{eq:def_r_growth}, for any $p < r$ there exists $K_{r,p} < \infty$ such that
\begin{align}
\En \left|\sup_{f \in \F} \sum_{t=1}^n \epsilon_t f(\x_t)\right|  \le K_{r,p}\ \Es{}{\left(\sum_{t=1}^n \sup_{f \in \F}  \left|f(\x_t)\right|^p\right)^{1/p}} ~. 
\end{align}
Further, if $\F\subseteq [-1,1]^\X$ and \eqref{eq:def_r_growth} holds with constant $D/2$, then
	\begin{align}
		\En \left| \sup_{f \in \F} \sum_{t=1}^n \epsilon_t f(\x_t) \right| & \le 32 D\ \log n\   \En \left[ \left(\sum_{t=1}^{n}\sup_{f \in \F} |f(\x_t)|^r \right)^{1/r}\right]   + \phi_n 
	\end{align}
	where $\phi_n \deq \frac{64 D \sqrt{n} \log n}{n^{D^2 \log n}}$ is a negligible term.
\end{lemma}

The second part of the proof of Lemma~\ref{lem:p_moment_lemma} uses the amplification idea of the previous section. 

Using Lemma~\ref{lem:sufficient_for_abs_regret}, we can now conclude existence of prediction strategies whose regret is controlled by sequence-dependent variance. This greatly extends the scope of available variance-type bounds in the online learning literature where results in this direction have been obtained for either finite or linear classes. 
\begin{corollary}
	\label{cor:final_results}
	Let $\F\subseteq [-1,1]^\X$ and $r\in(1,2]$. If \eqref{eq:def_r_growth} holds with constant $D/2$, then there exists a prediction strategy $(\pred_t)$ such that
$$
\sum_{t=1}^n (- \pred_t \cdot y_t) - \inf_{f \in \F} \sum_{t=1}^n (-f(x_t) \cdot y_t) \le 32 \cdot D \log_2(n)  \left(\sum_{t=1}^{n}\sup_{f \in \F} |f(x_t)|^r \right)^{1/r}  + \phi_n
$$
for any sequence of $(x_t,y_t)_{t=1}^n$ (equivalently, \eqref{eq:abs_regret} holds). 
\end{corollary}

In addition to being a novel result in the online learning domain, the above corollary serves as an amplification step to boost the in-expectation of bound of Lemma~\ref{lem:p_moment_lemma} to a high probability statement. We then invoke Corollary~\ref{cor:prob_symmetrization} and Lemma~\ref{lem:balance_alpha}  to prove the following theorem.

\begin{theorem}\label{thm:unif}
 	Let $Z_1,\ldots,Z_n$ be a stochastic process with values in $\cZ$ and let $Z'_1,\ldots,Z'_n$ be a tangent sequence. Let $\G\subseteq [-1,1]^\cZ$, $r\in(1,2]$, and define the $r$-variation as
	\begin{align}
		\Var_r = \En_{Z_{1:n}'}\sum_{t=1}^n\sup_{g \in \G} (g(Z_t)-g(Z'_t))^r ~.
	\end{align}
	If \eqref{eq:def_r_growth} holds for $\G$ with constant $D/2$, then with probability at least $1-e\log(n)\exp(-2u^2)$
	\begin{align*}
		&\sup_{g\in\G} \sum_{t=1}^n (g(Z_t)-\En_{t-1} g(Z_t)) \leq 256 D \log_2(n)  \Var_r^{1/r} + u\cdot 8\sqrt{\Var_2+1} + 8\phi_n .
	\end{align*}
\end{theorem}
We remark that the tail bound can be viewed as a ratio inequality (see \cite{pena2008self,GinKol06}) of the form \eqref{eq:banach_tail}, where the deviations are scaled by the square root of the variance.

\section{Finer control via per-function variation}
\label{sec:per_func}

From the point of view of the previous section (and Theorem~\ref{thm:equivalence_type}), all classes with sequential Rademacher complexity growth $n^{1/2}$ are treated equally. However, classes with such a growth can be as simple as a set consisting of two functions, or as complex as a set of linear functions indexed by a ball in the infinite-dimensional Hilbert space. In this section, a different complexity measure will be used for the regime when the $n^{1/2}$ growth hides the difference in complexity. This measure will be given by sequential covering numbers (and, as a consequence, by the offset Rademacher complexity). In the regime $\alpha^{-q}$, $q\in[0,2]$, for the growth of sequential entropy, we exhibit a finer analysis of the variation term that allows part of the variance to be adapted to the function. 
 
Let $q\in(0,2]$. We say that a class $\F\subseteq [-1,1]^\X$ has the $\gamma^{-q}$ growth (as $\gamma$ decreases) of sequential entropy if there is a constant $C$ such that for all $\gamma\in(0,1]$,
$$
\log \mathcal{N}_2(\F,\gamma,n) \leq C\gamma^{-q}.
$$
As for sequential Rademacher complexity, it is easy to check that the class $\G$ and the derived class of functions $(z,z')\mapsto f(z,z')=g(z)-g(z')$ have the same growth of sequential entropy. Moreover, this growth controls the rate of growth of the offset Rademacher complexity, as shown in \cite{RakSri14nonparam}. In particular, for the finite function class,
$$\Rad_n^{\mrm{off}}(\F, 1, \alpha) \leq \frac{8\log |\F|}{\alpha},$$
while for a parametric class of ``dimension'' $d$ (such that $\cN(\F,\gamma,n)\leq (C'/\gamma)^d$ for some $C'>0$),
$$\Rad_n^{\mrm{off}}(\F, 1, \alpha) \leq \frac{Cd\log (n)}{\alpha},$$
and for a class with sequential entropy growth $q\in(0,2)$,  
$$\Rad_n^{\mrm{off}}(\F, 1, \alpha) \leq C \alpha^{- \frac{2-q}{2+q}} n^{\frac{q}{2+q}}$$ for some absolute constant $C$ (the bound gains an extra logarithmic factor at $q=2$). In this last nonparametric regime, Corollary~\ref{cor:dev_above_off_rad} implies that for any $u>0$,
$$
P\Bigg(\sup_{f \in \F}\left\{  \sum_{t=1}^n \epsilon_t f(\x_t) - \frac{\alpha}{2} f(\x_t)^2 \right\} - C \alpha^{-\frac{2-r}{2+r}} n^{\frac{r}{2+r}} >   u \Bigg)   \leq \exp\left\{-\alpha u\right\},
$$ 
and the analogous statements hold for the finite and parametric cases. As the next Theorem shows, the offset Rademacher complexity $\Rad_n^{\mrm{off}}$ brings out (for smaller classes) the finer complexity control obscured by the sequential Rademacher complexity which only provides $\Omega(n^{1/2})$ bounds.

\begin{theorem}
	\label{thm:perfunction}
	Let $Z_1,\ldots,Z_n$ be a discrete-time process with values in $\cZ$ and let $Z'_1,\ldots,Z'_n$ be a tangent sequence. Let $\G\subseteq [-1,1]^\cZ$ and define function-dependent variance as
	\begin{align}
		\Var_2(g) = \En_{Z_{1:n}'}\sum_{t=1}^n (g(Z_t)-g(Z'_t))^2.
	\end{align}
	If $\G$ exhibits an $\gamma^{-q}$ growth of sequential entropy, then there exists a constant $C$ such that for any $u>0$, with probability at least $1-e\log (n)\exp\{-u^2\}$,
		\begin{align}
			\sup_{g\in\G} \sum_{t=1}^n (g(Z_t)-\En_{t-1} g(Z_t)) &\leq  Cn^{\frac{q}{4}}\left(\Var_2(g) +2\right)^\frac{2-q}{4} + u\cdot 2\sqrt{2}\sqrt{\Var_2(g) +2}.
		\end{align}
		If $\G$ is finite, with the same probability it holds that
		\begin{align}
			\sup_{g\in\G} \sum_{t=1}^n (g(Z_t)-\En_{t-1} g(Z_t)) &\leq  C\sqrt{\log |\G|}\sqrt{\Var_2(g) +2} + u\cdot 2\sqrt{2}\sqrt{\Var_2(g) +2},
		\end{align}
		while for the parametric case, 
		\begin{align}
			\sup_{g\in\G} \sum_{t=1}^n (g(Z_t)-\En_{t-1} g(Z_t)) &\leq  C\sqrt{d\log n}\sqrt{\Var_2(g) +2} + u\cdot 2\sqrt{2}\sqrt{\Var_2(g) +2}.
		\end{align}
\end{theorem}

The finite and parametric cases can be thought of as a ``$q=0$'' regime. Here, we have a  bound that depends on $n$ at most logarithmically. On the other hand, for $q\geq 2$ the term $n^{\frac{q}{4}} (\Var_2+1)^{\frac{2-q}{4}}$ is replaced with $n^{1-1/q}$, without any per-function adaptivity (as studied in the previous section). Between these two regimes, we obtain an interpolation, whereby the $1/2$ power is split into a non-adaptive part $n^{\frac{q}{4}}$ and the adaptive part $(\Var_2+1)^{\frac{2-q}{4}}$. This constitutes a finer analysis of classes with martingale type 2.

We may compare the bound of Theorem~\ref{thm:perfunction} in the finite case to the in-expectation bound of \cite{massart2013around} in terms of ``weak variance'' for i.i.d. zero mean random variables $Z_{1},\ldots, Z_n \in \reals^d$:
$$\En\left[\max_{j\leq d} | \sum_{t=1}^n \epsilon_t Z_{t,j} | \right]\leq \sqrt{2\ln (2d) \En \max_{j\leq d} \sum_{t=1}^n Z_{t,j}^2}.$$
In contrast to this bound, Theorem~\ref{thm:perfunction} matches the coordinate $j$ on the left-hand side to the variance of the $j$th coordinate on the right-hand side. Further, our bound holds for martingale difference sequences rather than i.i.d. random vectors. Finally, Theorem~\ref{thm:perfunction} holds well beyond the finite case.

\input{conclude}

\appendix

\input{appendix}

\bibliographystyle{plain}
\bibliography{paper}
\end{document}

%% file: macros.tex
\include{multiminimax}

\DeclareMathOperator*{\Ex}{\vphantom{p}\mathbb{E}}
\let\inf\undef
\DeclareMathOperator*{\inf}{\vphantom{p}inf}
\let\sup\undef
\DeclareMathOperator*{\sup}{\vphantom{p}sup}

\newcommand{\inner}[1]{\left\langle #1 \right\rangle}

\newcommand{\cA}{\ensuremath{\mathcal A}}
\newcommand{\cB}{\ensuremath{\mathcal B}}

\newcommand{\E}{\ensuremath{\mathbb E}} 
\newcommand{\F}{\ensuremath{\mathcal F}}

\newcommand{\cN}{\ensuremath{\mathcal N}} 
\newcommand{\cR}{\ensuremath{\mathcal R}}

\newcommand{\cX}{\ensuremath{\mathcal X}}

\newcommand{\cZ}{\ensuremath{\mathcal Z}}

\newcommand{\mbb}[1]{\mathbb{#1}}
\newcommand{\mbf}[1]{\mathbf{#1}}
\newcommand{\mc}[1]{\mathcal{#1}}
\newcommand{\mrm}[1]{\mathrm{#1}}

\newcommand{\reals}{\ensuremath{\mathbb R}}

\newcommand{\ind}[1]{{\bf I}\left\{#1\right\}}
\def\deq{\triangleq}

\newcommand{\Rad}{\mc{Rad}}

\newcommand{\loss}{\ensuremath{\boldsymbol\ell}}

\newcommand{\norm}[1]{\left\|#1\right\|}
\newcommand{\fat}{\text{\rm fat}}

\newcommand{\pred}{\boldsymbol{\widehat{y}}}

\newcommand{\s}{\ensuremath{\mathbf s}} 
  
\newcommand{\x}{\ensuremath{\mathbf x}} 
\newcommand{\y}{\ensuremath{\mathbf y}} 
\newcommand{\z}{\ensuremath{\mathbf z}} 
\newcommand{\argmin}[1]{\underset{#1}{\mrm{argmin}} \ }

\newcommand{\Es}[2]{\mbb{E}_{#1}\left[ #2 \right]}



%% file: header.tex
\usepackage{amsmath,amsfonts,amssymb,amsthm}
\usepackage{graphicx}
\usepackage{color}
\usepackage{tikz}
\usetikzlibrary{arrows,shapes}
\usetikzlibrary{patterns}

\newtheorem{theorem}{Theorem}
\newtheorem{lemma}[theorem]{Lemma}

\newtheorem{definition}{Definition}

\newtheorem{remark}{Remark}[section]
\newtheorem{corollary}[theorem]{Corollary}

\newcommand{\ip}[2]{{\left\langle{#1},{#2}\right\rangle}}

%% file: notation.tex
\renewcommand{\v}{\ensuremath{\mathbf{v}}}

\newcommand{\G}{\mathcal{G}}

\newcommand{\X}{\mathcal{X}}

\renewcommand{\v}{\mathbf{v}}

\renewcommand{\Rad}{\mc{R}}

\renewcommand{\X}{\mc{X}}

\renewcommand{\z}{\mbf{z}}
\newcommand{\En}{\mbb{E}}

\renewcommand{\F}{\mc{F}}

%% file: conclude.tex

\section{Some Open Questions}
The following are a few open-ended questions raised by this work:
\begin{enumerate}
 \item In the definition of martingale type, can we replace $ \E{\left( \sum_{t=1}^n \E'_{t-1}\sup_{g \in \G}  \left|g(Z_t) - g(Z'_t)\right|^p \right)^{1/p}}$ with $\E{\left(\sum_{t=1}^n \sup_{g\in\G }\left|g(Z_t) - \Es{t-1}{g}\right|^p\right)^{1/p}}$ and reach the same conclusions? The latter version of variation is closer to the generalization of the martingale type for Banach spaces. 
\item If for some $r \in (1,2]$, sequential Rademacher complexity exhibits $n^{1/r}$ growth rate, then does $\G$ have martingale type $r$? Currently, we only prove martingale type $p$ for any $p<r$. For the case of Banach spaces (linear $g$), the above question is answered in the positive in the work of Pisier \cite{Pisier75}. However, the result of \cite{Pisier75} relies on the notions of uniform convexity or uniform smoothness which are specific to linear functionals and Banach spaces. 
\item Is it possible to get a mix of uniform and per-funtion variance for general function classes with martingale type $2$? In Section \ref{sec:per_func}, for martingale type $2$ we prove a finer control through per function variance. A natural question is whether one can replace the $n$-dependent part by uniform variance terms thus giving a mix of per-function and uniform variance in the same bound.
\end{enumerate}

%% file: appendix.tex
\section{Proofs}

\begin{lemma}
	\label{lem:gd}
	The update in \eqref{eq:gradient_descent} satisfies
	$$\forall z_1,\ldots,z_n\in \mathcal{B}, ~~~ \sum_{t=1}^n \inner{\pred_t-f, z_t} \leq \sqrt{n}.$$
\end{lemma}
\begin{proof}[\textbf{Proof of Lemma~\ref{lem:gd}}]
The following two-line proof is standard. By the property of a projection,
$$\norm{\pred_{t+1}-f}^2 =  \norm{\text{Proj}_{\cB}(\pred_{t}-n^{-1/2}z_t)-f}^2 \leq \norm{(\pred_{t}-n^{-1/2}z_t)-f}^2 = \norm{\pred_{t}-f}^2 + \frac{1}{n}\norm{z_t}^2 - 2n^{-1/2}\inner{\pred_t-f, z_t}.$$
Rearranging,
$$2n^{-1/2}\inner{\pred_t-f, z_t}  \leq  \norm{\pred_{t}-f}^2 - \norm{\pred_{t+1}-f}^2 +  \frac{1}{n}\norm{z_t}^2.$$
Summing over $t=1,\ldots,n$ yields the desired statement.
\end{proof}

\begin{lemma}
	\label{lem:BDG}
	With the notation of Lemma~\ref{thm:banach_space},
	$$\En \max_{s=1,\ldots,n}  \norm{ \sum_{t=1}^s  Z_t}  \leq  \left(2.5 R_{\max}+\sqrt{3}\right) \En\sqrt{V_n} + 2.5 R_{\max}\ .
	$$
\end{lemma}
\begin{proof}[\textbf{Proof of Lemma~\ref{lem:BDG}}]
	Because of the ``anytime'' property of the regret bound and the strategy definition, we can write \eqref{eq:banach_space_regret_norm} as 
	\begin{align}
		\label{eq:uniform_in_time_norm}
		 \max_{s=1,\ldots,n} \left\{ \norm{ \sum_{t=1}^s  Z_t} - \sum_{t=1}^s   \inner{\pred_t, Z_t} \right\} \leq 2.5 R_{\max} \left(\sqrt{V_n} + 1\right)
	\end{align}
	simply because the right-hand side is largest for $s=n$. Sub-additivity of $\max$ implies
	\begin{align}
		 \max_{s=1,\ldots,n}  \norm{ \sum_{t=1}^s  Z_t}  - 2.5 R_{\max} \left(\sqrt{V_n} + 1\right) \leq \max_{s=1,\ldots,n} \sum_{t=1}^s   \inner{\pred_t, Z_t} . 
	\end{align}
	By the Burkholder-Davis-Gundy inequality (with the constant from \cite{burkholder2002best}),
	\begin{align}
		\En \max_{s=1,\ldots,n} \sum_{t=1}^s   \inner{\pred_t, Z_t} \leq \sqrt{3}\En \left(\sum_{t=1}^n   \inner{\pred_t, Z_t}^2 \right)^{1/2} \leq \sqrt{3}\En\sqrt{V_n} \ .
	\end{align}
	In view of \eqref{eq:uniform_in_time_norm}, we conclude the statement.
\end{proof}

\begin{proof}[\textbf{Proof of Lemma~\ref{lem:pred_seq_regret}}]
	Because of the update form, 
	$$\forall f\in\F, ~~\inner{\pred_{t+1} - f, z_t} \leq \frac{1}{\eta_t}\left( D_\cR(f,\pred_{t})-D_\cR(f, \pred_{t+1}) - D_\cR(\pred_{t+1},\pred_{t})\right).$$
	Summing over $t=1,\ldots,n$,
	\begin{align*}
		\sum_{t=1}^n \inner{\pred_{t+1} - f, z_t} 
		&\leq \eta_1^{-1}D_\cR(f, \pred_1) + \sum_{t=2}^n (\eta_t^{-1}-\eta_{t-1}^{-1}) D_\cR(f,\pred_{t}) - \sum_{t=1}^n \eta_t^{-1} D_\cR(\pred_{t+1},\pred_{t}) \\
		&\leq \eta_1^{-1}R^2_{\max} + \sum_{t=2}^n (\eta_t^{-1}-\eta_{t-1}^{-1}) R^2_{\max} - \sum_{t=1}^n \frac{\eta_t^{-1}}{2}\norm{\pred_{t+1}-\pred_{t}}_*^2 \\
		&\leq R^2_{\max} (\eta_1^{-1}+\eta_n^{-1}) - \sum_{t=1}^n \frac{\eta_t^{-1}}{2}\norm{\pred_{t+1}-\pred_{t}}_*^2,
	\end{align*}
	where we used strong convexity of $\cR$ and the fact that $\eta_t$ is nonincreasing. Next, we write
	\begin{align*}
		\sum_{t=1}^n \inner{\pred_{t} - f, z_t} &= \sum_{t=1}^n \inner{\pred_{t+1} - f, z_t} + \sum_{t=1}^n \inner{\pred_{t} - \pred_{t+1}, z_t} 
	\end{align*}
	and upper bound the second term by noting that
	$$\inner{\pred_{t} - \pred_{t+1}, z_t}\leq \norm{\pred_{t} - \pred_{t+1}}_*\cdot \norm{z_t} \leq \frac{\eta_{t}^{-1}}{2} \norm{\pred_{t} - \pred_{t+1}}_*^2 + \frac{\eta_{t}}{2}\norm{z_t}^2.$$
	Combining the bounds,
	\begin{align}
		\label{eq:md_1}
		\sum_{t=1}^n \inner{\pred_{t} - f, z_t} &\leq R^2_{\max} (\eta_1^{-1}+\eta_n^{-1}) + \sum_{t=1}^n \frac{\eta_{t}}{2}\norm{z_t}^2.
	\end{align}
	Now observe that 	
	\begin{align}
		\textstyle \eta_t = R_{\max}\min\left\{1, \frac{\sqrt{\sum_{s=1}^{t} \norm{z_s}^2}-\sqrt{\sum_{s=1}^{t-1} \norm{z_s}^2}}{\norm{z_t}^2}\right\}  
	\end{align}
	and thus the second term in \eqref{eq:md_1} is upper bounded as
	$$\sum_{t=1}^n \frac{\eta_{t}}{2}\norm{z_t}^2 \leq \frac{R_{\max}}{2}\sqrt{\sum_{s=1}^n \norm{z_s}^2}.$$
	For the first term, we use $\eta_1^{-1} = R_{\max}^{-1}$ and 
	$$\eta_n^{-1} \leq R_{\max}^{-1}\max\left\{1, 2\sqrt{\sum_{s=1}^{t} \norm{z_s}^2} \right\} $$
	Concluding,
	\begin{align}
		\label{eq:md_2}
		\sum_{t=1}^n \inner{\pred_{t} - f, z_t} &\leq R_{\max} \left(2+2.5\sqrt{\sum_{s=1}^{t} \norm{z_s}^2}\right) .
	\end{align}
\end{proof}

\begin{proof}[\textbf{Proof of Lemma~\ref{eq:AB-condition-delapena}}]
	We have
	\begin{align*}
		&\En_{t-1}\exp\left\{\lambda A - \lambda^2 B^2/2 \right\} \\
		&=\En_{t-1}\exp\left\{\lambda \sum_{t=1}^n \inner{\pred_t, Z_t-\En_{t-1} Z'_t} - 2\lambda^2 \sum_{t=1}^n (\norm{Z_t}^2 + \En_{t-1}\norm{Z'_t}^2) \right\} \\
		&\leq \En_{t-1}\exp\left\{\lambda \sum_{t=1}^n \inner{\pred_t, Z_t-Z'_t} - 2\lambda^2 \sum_{t=1}^n (\norm{Z_t}^2 + \norm{Z'_t}^2) \right\} \\
		&\leq \En_{t-1}\En_\epsilon \exp\left\{\lambda \sum_{t=1}^n \epsilon_t\inner{\pred_t, Z_t-Z'_t} - 2\lambda^2 \sum_{t=1}^n (\norm{Z_t}^2 + \norm{Z'_t}^2) \right\} .
	\end{align*}
	Since $\exp$ is a convex function,
	\begin{align*}
		&\En_{t-1}\En_\epsilon\exp\left\{\frac{1}{2}\left(2\lambda \sum_{t=1}^n \epsilon_t\inner{\pred_t, Z_t} - 4\lambda^2 \sum_{t=1}^n \norm{Z_t}^2 \right) + \frac{1}{2}\left(2\lambda \sum_{t=1}^n \inner{\pred_t, -Z'_t} - 4\lambda^2 \sum_{t=1}^n \norm{Z'_t}^2 \right) \right\} \\
		&\leq \frac{1}{2}\En_{t-1}\En_\epsilon\exp\left\{2\lambda \sum_{t=1}^n \epsilon_t\inner{\pred_t, Z_t} - 4\lambda^2 \sum_{t=1}^n \norm{Z_t}^2 \right\}+ \frac{1}{2}\En_{t-1}\En_\epsilon\exp\left\{ 2\lambda \sum_{t=1}^n \epsilon_t\inner{\pred_t, -Z'_t} - 4\lambda^2 \sum_{t=1}^n \norm{Z'_t}^2 \right\}\\
		&= \En_{t-1}\En_\epsilon\exp\left\{2\lambda \sum_{t=1}^n \epsilon_t\inner{\pred_t, Z_t} - 4\lambda^2 \sum_{t=1}^n \norm{Z_t}^2 \right\}\\
		&\leq \En_{t-1}\exp\left\{4\lambda^2 \sum_{t=1}^n |\inner{\pred_t, Z_t}|^2 - 4\lambda^2 \sum_{t=1}^n \norm{Z_t}^2 \right\}\\
		&\leq 1 
	\end{align*}
	since $\|\pred_t\|_*\leq 1$.
\end{proof}

\begin{proof}[\textbf{Proof of Theorem~\ref{thm:equivalence_type}}]
	Let $X_1,\ldots,X_n$ be a discrete time process. We have
\begin{align*}
& \E{\left|\sup_{f \in \F}\left\{\sum_{t=1}^n (f(X_t) - \Es{t-1}{f(X_t)}) \right\}\right| - C \left(\sum_{t=1}^n \En'_{t-1}\sup_{f \in \F} |f(X_t) - f(X'_t) |^p\right)^{1/p}}\\
& ~~~~~\le \multiminimax{\sup_{p_t} \Ex_{X_t \sim p_t}}_{t=1}^n \left[ \sup_{f \in \F}\left\{\sum_{t=1}^n (f(X_t) - \Es{X'_t \sim p_t}{f(X'_t)}) \right\} - C \left(\sum_{t=1}^n \En_{X'_t\sim p_t}\sup_{f \in \F} |f(X_t) - f(X'_t) |^p\right)^{1/p} \right]
\end{align*}
where $\multiminimax{\sup_{p_t} \Ex_{X_t \sim p_t}}_{t=1}^n$ stands for repeated application of the operators: $\sup_{p_1}\Ex_{X_1}\ldots \sup_{p_n}\Ex_{X_n}$. By Jensen's inequality, we upper bound the above expression by
\begin{align*}
&\multiminimax{\sup_{p_t} \En_{X_t, X'_t \sim p_t}}_{t=1}^n \left[ \left|\sup_{f \in \F}\left\{\sum_{t=1}^n (f(X_t) - f(X'_t)) \right\}\right| - C \left(\sum_{t=1}^n \sup_{f \in \F} |f(X_t) - f(X'_t) |^p\right)^{1/p} \right] ~.
\end{align*}
Introducing independent Rademacher random variables $\epsilon_1,\ldots,\epsilon_n$, the preceding expression is equal to
\begin{align*}
	&\multiminimax{\sup_{p_t} \En_{X_t, X'_t \sim p_t} \En_{\epsilon_t} }_{t=1}^n \Bigg[ \left|\sup_{f \in \F}\left\{\sum_{t=1}^n \epsilon_t(f(X_t) - f(X'_t)) \right\}\right| - C \left(\sum_{t=1}^n \sup_{f \in \F} |f(X_t) - f(X'_t) |^p\right)^{1/p} \Bigg]\\
	&\leq \multiminimax{\sup_{x_t, x'_t} \En_{\epsilon_t}}_{t=1}^n \Bigg[ \left|\sup_{f \in \F}\left\{\sum_{t=1}^n \epsilon_t (f(x_t) - f(x'_t)) \right\}\right| - C \left(\sum_{t=1}^n \sup_{f \in \F} |f(x_t) - f(x'_t) |^p\right)^{1/p} \Bigg] ~.
\end{align*}
The latter expression may be written as
\begin{align}
	\label{eq:doubletree}
	\sup_{\x,\x'}\En \Bigg[ \left|\sup_{f \in \F}\left\{\sum_{t=1}^n \epsilon_t (f(\x_t) - f(\x'_t)) \right\}\right| - C \left(\sum_{t=1}^n \sup_{f \in \F} |f(\x_t) - f(\x'_t) |^p\right)^{1/p} \Bigg]
\end{align}
with a supremum ranging over predictable processes $\x=(\x_1,\ldots,\x_n)$ and $\x'=(\x'_1,\ldots,\x'_n)$, each $\x_t,\x'_t:\{\pm1\}^{t-1}\to\cX$. Now define the function class $\G \subset \reals^{\X \times \X}$ as follows:
$$
\G = \{ (x,x') \mapsto f(x) - f(x') : f \in \F\}~.
$$
Trivially, \eqref{eq:doubletree} can be written with this notation as
\begin{align*}
 &\sup_{\x,\x'}\En \Bigg[ \left|\sup_{g \in \G}\left\{\sum_{t=1}^n \epsilon_t g(\x_t, \x'_t) \right\}\right| - C \left(\sum_{t=1}^n \sup_{g \in \G} |g(\x_t,\x'_t) |^p\right)^{1/p} \Bigg] ~. 
\end{align*}
However the complexity of $\G$ is not much larger than that of $\F$:
$$\Rad_n(\G; (\x,\x')) = \En \left| \sup_{g \in \G} \sum_{t=1}^n \epsilon_t g(\x_t, \x'_t)  \right| = \En \left| \sup_{f \in \F} \sum_{t=1}^n \epsilon_t (f(\x_t)-f(\x'_t))  \right| \leq  \Rad_n(\F; \x)+\Rad_n(\F;\x').$$
The first part of the theorem is concluded by applying Lemma~\ref{lem:p_moment_lemma} to the class $\G$.

To prove the second part, we modify the lower bound construction in \cite[Theorem 2]{RakSriTew14}.
Assume that we are given a predictable process $\x$ of length $n$ and that $\x_0$ is any one of the $2^{n}-1$ values in the image of $\x$. Since $\Es{\epsilon}{\left|\sum_{t=1}^n \epsilon_t \right|} \le \sqrt{n}$, we have that
\begin{align*}
 \Rad_n(\F; \x) &= \Es{\epsilon}{ \left| \sup_{f\in\F} \sum_{t=1}^n \epsilon_t f(\x_t)  \right|}\\
 &\le \Es{\epsilon}{\left| \sup_{f\in\F} \sum_{t=1}^n \epsilon_t f(\x_t)  \right|}  - \sup_{f \in \F}|f(\x_0)| \ \Es{\epsilon}{\left|\sum_{t=1}^n \epsilon_t \right|} + \sup_{f \in \F}|f(\x_0)| \ \sqrt{n} \\
 &\le \Es{\epsilon}{\left| \sup_{f\in\F} \sum_{t=1}^n \epsilon_t f(\x_t)  \right|} - \Es{\epsilon}{\left|\sup_{f \in \F} \sum_{t=1}^n \epsilon_t  f(\x_0)\right|}  + \sup_{f \in \F}|f(\x_0)| \ \sqrt{n}\\
 &\le \Es{\epsilon}{\left| \sup_{f\in\F} \sum_{t=1}^n \epsilon_t f(\x_t)  - \sup_{f \in \F} \sum_{t=1}^n \epsilon_t  f(\x_0) \right| }   + \sup_{f \in \F}|f(\x_0)| \ \sqrt{n}\\
 &\le \Es{\epsilon}{\left| \sup_{f\in\F} \sum_{t=1}^n \epsilon_t \left(f(\x_t) - f(\x_0)\right)  \right|}    + \sup_{f \in \F}|f(\x_0)| \ \sqrt{n}\\
 &\le 2 \Es{\epsilon}{\left| \sup_{f\in\F} \sum_{t=1}^n \frac{f(\x_t) + f(\x_0)}{2} -  \frac{1-\epsilon_t}{2} f(\x_t) - \frac{1 + \epsilon_t}{2} f(\x_0) \right|}    + \sup_{f \in \F}|f(\x_0)| \ \sqrt{n}
\end{align*}
Now consider the joint distribution over $X_1,\ldots,X_n$ such that, for every $t \in [n]$,  $P(X_t = \x_0| \epsilon_{t-1},\epsilon_t =1) = 1$ and $P(X_t = \x_t(\epsilon_{1:t-1})| \epsilon_{t-1},\epsilon_t =0) = 1$. Under this distribution, we can rewrite the above inequality as
\begin{align*}
 \Rad_n(\F; \x) & \le 2\,\Es{}{\left| \sup_{f\in\F} \sum_{t=1}^n \left(\Es{t-1}{f(X_t)} -  f(X_t)\right)\right|}    + \sup_{f \in \F}|f(\x_0)| \ \sqrt{n}.
\end{align*}
Since $\F$ is of type $r$, taking $\epsilon'_1,\ldots,\epsilon'_n$ to be an independent Rademacher sequence, we further bound the above term as
\begin{align*}
&2 C\,\Es{}{ \left(\sum_{t=1}^n  \E'_{t-1}\sup_{f\in\F}\left|f(X_t) -  f(X'_t)\right|^r\right)^{1/r}}    + \sup_{f \in \F}|f(\x_0)| \ \sqrt{n}\\
& \leq 2 C\,\Es{}{ \left(\sum_{t=1}^n  \sup_{f\in\F}\left|  \frac{1-\epsilon_t}{2} f(\x_t) + \frac{1 + \epsilon_t}{2} f(\x_0) - \frac{1-\epsilon'_t}{2} f(\x_t) - \frac{1 + \epsilon'_t}{2} f(\x_0) \right|^r\right)^{1/r}}    + \sup_{f \in \F}|f(\x_0)| \ \sqrt{n}\\
& \le 4 C\, \left( n \left(\sup_{f \in \F, t \le n, \epsilon \in \{\pm1\}^n} \left|f(\x_t)\right|^{r} + \sup_{f \in \F} |f(\x_0)|^r\right) \right)^{1/r}   + \sup_{f \in \F}|f(\x_0)| \ \sqrt{n}.
\end{align*}
Since $\x_0$ is one of the elements of the tree $\x$, we further upper bound the expression by
\begin{align*}
8 C\, n^{1/r}\left(\sup_{f \in \F, t \le n, \epsilon \in \{\pm1\}^n} \left|f(\x_t)\right|\right)    + \sup_{f \in \F}|f(\x_0)| \ \sqrt{n} \le 16 C\, n^{1/r}\left(\sup_{f \in \F, t \le n, \epsilon \in \{\pm1\}^n} \left|f(\x_t)\right|\right) ~.  
\end{align*}
In the last step we used the fact that $r \le 2$ and so $\sqrt{n} \le n^{1/r}$.
\end{proof}

\begin{proof}[\textbf{Proof of Lemma~\ref{lem:tail_comparison}}]
	We have 
	$$P(\xi\geq u)\leq \frac{\En \phi(\xi)}{\phi(u)}\leq \frac{\En \phi(\nu)}{\phi(u)}\leq  \frac{1}{\phi(u)}\left(\phi(0)+\int_0^\infty \phi'(x)P(\nu\geq x) dx \right).$$
	Choose $a = u - \mu^{-1}(1)$, where $\mu^{-1}$ is the inverse function. If $a<0$, the conclusion of the lemma is true since $\Gamma\geq 1$. In the case of $a\geq 0$, we have $\phi(0)=0$. The above upper bound becomes
	\begin{align*}
		P(\xi\geq u) &\leq \frac{\Gamma}{\phi(u)}\int_0^\infty \phi'(x)\exp\{- \mu(x)\} dx = \frac{\Gamma}{\phi(u)}\int_a^\infty \mu'(x)\exp\{- \mu(x)\} dx \\
		&= \frac{\Gamma}{\mu(u-a)} \left[ -\exp\{- \mu(x)\}\right]_a^\infty = \Gamma \exp\{- \mu(a)\} =  \Gamma \exp\{- \mu(u-\mu^{-1}(1))\}.
	\end{align*}	
	If $\mu(b)=c b$, we have 
	$$P(\xi\geq u)\leq \Gamma \exp\{-c(u-1/c) \} = \Gamma \exp\{1-c u\}.$$
	If $\mu(b)=c b^2$, we have
	$$P(\xi\geq u)\leq \Gamma \exp\{-c(u-1/\sqrt{c})^2 \} \leq \Gamma \exp\{-c u^2/4\}$$
	whenever $u \geq 2/\sqrt{c}$. If $u\leq 2/\sqrt{c}$, the conclusion is valid since $\Gamma\geq 1$.
\end{proof}

\begin{proof}[\textbf{Proof of Corollary~\ref{cor:prob_symmetrization}}]
Let 
$$\xi(Z_1,\ldots,Z_n, Z_1', \ldots,Z_n') = \sup_{g} \sum_{t=1}^n (g(Z_t)-g(Z_t')) - \tilde{B}(g; Z_1,Z_1',\ldots,Z_n, Z_n')$$
and
$$\nu(Z_1,\ldots,Z_n) = \sup_{g} \sum_{t=1}^n (g(Z_t)-\En_{t-1} g(Z_t')) - \En_{Z_{1:n}'} \tilde{B}(g; Z_1,Z_1',\ldots,Z_n, Z_n').$$
Then for any convex $\phi:\reals\to\reals$,
$$\En\phi(\nu) \leq \En\phi(\xi)$$
using convexity of the supremum. The problem is now reduced to obtaining tail bounds for
\begin{align*}
	&P\left(\sup_{f} \sum_{t=1}^n (g(Z_t)-g(Z_t')) - \tilde{B}(g; Z_1,Z_1',\ldots,Z_n, Z_n')>u\right) .
\end{align*}
Write the probability as 
$$\En\ind{\xi(Z_1,\ldots,Z_n, Z_1', \ldots,Z_n')>u}.$$
We now proceed to replace the random variables from $n$ backwards with a dyadic filtration. Let us start with the last index. Renaming $Z_n$ and $Z_n'$ we see that
{\small\begin{align*}
	&\En\ind{\sup_{g} \sum_{t=1}^n (g(Z_t)-g(Z_t')) - \tilde{B}(g; Z_1,Z_1',\ldots,Z_n,Z_n') >u} \\
	&=\En\ind{\sup_{g} \sum_{t=1}^{n-1} (g(Z_t)-g(Z'_t)) + (g(Z'_n)-g(Z_n)) - \tilde{B}(g; Z_1,Z_1',\ldots,Z_n,Z_n')>u}\\
	&=\En\En_{\epsilon_n}\ind{\sup_{g} \sum_{t=1}^{n-1} (g(Z_t)-g(Z'_t)) + \epsilon_n(g(Z_n)-g(Z'_n)) - \tilde{B}(g; Z_1,Z_1',\ldots,Z_n,Z_n')>u}\\
	&\leq \En\sup_{z_n,z_n'}\En_{\epsilon_n}\ind{\sup_{g} \sum_{t=1}^{n-1} (g(Z_t)-g(Z'_t)) + \epsilon_n(g(z_n)-g(z'_n)) - \tilde{B}(g; Z_1,Z_1',\ldots,Z_{n-1},Z_{n-1}', z_n,z_n')>u}.
\end{align*}}
Proceeding in this manner for step $n-1$ and back to $t=1$, we obtain an upper bound of
\begin{align*}
	&\sup_{z_1,z_1'}\En_{\epsilon_1}\ldots \sup_{z_n,z_n'}\En_{\epsilon_n}\ind{\sup_{g} \sum_{t=1}^{n} \epsilon_t(g(z_t)-g(z'_t)) - \tilde{B}(g; z_1,z_1',\ldots,z_n,z_n')>u}\\
	&=\sup_{\x} \En \ind{\sup_{g} \sum_{t=1}^{n} \epsilon_t f_g(\x_t) - B(g; \x_1,\ldots,\x_n)>u}.
\end{align*}
\end{proof}

\begin{proof}[\textbf{Proof of Lemma~\ref{lem:per_tree_abs_regret}}]
To check the desired statement \eqref{eq:abs_regret_per_tree} for the given predictable process $\x$, we verify that
\begin{align}\label{eq:show}
 \multiminimax{\inf_{\hat{y}_t} \max_{y_t \in \{\pm 1\}}}_{t=1}^n \left[\sum_{t=1}^n |\hat{y}_t - y_t| - \inf_{f \in \F} \left\{ \sum_{t=1}^n |f(\x_t(y_{1:t-1})) - y_t| + B(f; \x) \right\} \right] \le 0 
\end{align}
where each infimum is taken over the set $\left\{\hat{y}_t : |\hat{y}_t| \le \sup_{f \in \F} |f(\x_t(y_{1:t-1}))| \right\}$. To this end, 
\begin{align*}
& 2\times \multiminimax{\inf_{\hat{y}_t} \max_{y_t \in \{\pm 1\}}}_{t=1}^n \left[\sum_{t=1}^n |\hat{y}_t - y_t| - \inf_{f \in \F} \left\{ \sum_{t=1}^n |f(\x_t(y_{1:t-1})) - y_t| + B(f; \x)\right\}\right] \\
& = \multiminimax{\inf_{\hat{y}_t} \sup_{y_t \in \{\pm 1\}}}_{t=1}^n \left[\sum_{t=1}^n (- \hat{y}_t \cdot y_t) - \inf_{f \in \F} \left\{ \sum_{t=1}^n (- f(\x_t(y_{1:t-1})) \cdot y_t) + 2B(f;\x) \right\}\right] \\
& = \multiminimax{\sup_{p_t} \underset{y_t \sim p_t}{\En}}_{t=1}^n \left[\sum_{t=1}^n \inf_{\hat{y}_t} \left\{- \hat{y}_t \Es{}{y_t}\right\} - \inf_{f \in \F} \left\{ \sum_{t=1}^n - f(\x_t(y_{1:t-1})) \cdot y_t + 2B(f; \x) \right\}\right] 
\end{align*}
where $p_t$ ranges over distributions on $\{\pm1\}$. In the last step, we have used the minimax theorem, and then the technique that can be found, for instance, in \cite{AbeAgrBarRak09colt, RakSriTew14jmlr}. Next, we replace the infima by (sub)optimal choices corresponding to the value of $f$. This yields
\begin{align*}
	&\multiminimax{\sup_{p_t} \underset{y_t \sim p_t}{\En}}_{t=1}^n \left[\sup_{f \in \F}\left\{ \sum_{t=1}^n f(\x_t(y_{1:t-1})) \cdot y_t - \sup_{\hat{y}_t} \left\{ \hat{y}_t \Es{}{y_t}\right\} - 2B(f; \x) \right\} \right] \\
& \le \multiminimax{\sup_{p_t} \underset{y_t \sim p_t}{\En}}_{t=1}^n \left[\sup_{f \in \F}\left\{ \sum_{t=1}^n f(\x_t(y_{1:t-1})) \cdot ( y_t - \Es{}{y_t} ) - 2B(f; \x) \right\}  \right] \\
& \le \multiminimax{\sup_{p_t} \underset{y_t, y'_t \sim p_t}{\En} }_{t=1}^n \left[\sup_{f \in \F}\left\{ \sum_{t=1}^n f(\x_t(y_{1:t-1})) \cdot ( y_t - y'_t) - 2B(f; \x) \right\}\right] 
\end{align*}
where the last step is by Jensen's inequality. We further upper bound the above expression by 
\begin{align*}
& \multiminimax{\sup_{p_t} \underset{y_t, y'_t \sim p_t}{\En} \max_{y''_t} }_{t=1}^n \left[\sup_{f \in \F}\left\{ \sum_{t=1}^n f(\x_t(y''_{1:t-1})) \cdot ( y_t - y'_t) - 2B(f; \x) \right\}\right] 
\end{align*}
where $y''_t$ ranges over $\{\pm 1\}$. Since $y_t,y'_t$ can be renamed, we introduce the random signs 
\begin{align*}
& \multiminimax{ \sup_{p_t} \underset{y_t, y'_t \sim p_t}{\En} \underset{\epsilon_t}{\En} \max_{y''_t} }_{t=1}^n \left[\sup_{f \in \F}\left\{ \sum_{t=1}^n \epsilon_t  ( y_t - y'_t) f(\x_t(y''_{1:t-1})) - 2B(f; \x)\right\}\right] \\
&\leq \multiminimax{ \max_{y_t, y'_t} \underset{\epsilon_t}{\En} \max_{y''_t} }_{t=1}^n \left[\sup_{f \in \F}\left\{ \sum_{t=1}^n \epsilon_t  ( y_t - y'_t) f(\x_t(y''_{1:t-1})) - 2B(f; \x)\right\}\right] \\
& \le \multiminimax{\max_{b_t \in \{\pm 1\}} \underset{\epsilon_t}{\En} \max_{y_t} }_{t=1}^n \left[\sup_{f \in \F}\left\{ \sum_{t=1}^n 2 \epsilon_t  b_t f(\x_t(y_{1:t-1})) - 2B(f; \x) \right\}\right] ~.
\end{align*}
Since $b_t \epsilon_t$ has the same distribution as $\epsilon_t$ for any $b_t\in\{\pm1\}$, we write the above expression as
\begin{align*}
&\multiminimax{\underset{\epsilon_t}{\En} \max_{y_t} }_{t=1}^n \left[\sup_{f \in \F}\left\{ \sum_{t=1}^n 2 \epsilon_t  f(\x_t(y_{1:t-1})) - 2B(f; \x) \right\}\right] ~.
\end{align*}
To be consistent with the notation of predictable processes, we shift the numbering on $y$ by one:
\begin{align*}
	\multiminimax{\underset{\epsilon_t}{\En} \max_{y_{t+1}} }_{t=1}^n \left[\sup_{f \in \F}\left\{ \sum_{t=1}^n 2 \epsilon_t  f(\x_t(y_{2:t})) - 2B(f; \x) \right\}\right]
	& = \sup_{\y} \Es{\epsilon}{\sup_{f \in \F} \sum_{t=1}^n 2 \epsilon_t  f(\x_t(\y_{2:t}(\epsilon))) - 2B(f; \x) }\\
	&\leq \sup_{\y} \Es{\epsilon}{\sup_{f \in \F} \sum_{t=1}^n 2 \epsilon_t  f(\x_t(\y_{2:t}(\epsilon))) - 2B(f; \x\circ\y) }
\end{align*}
by \eqref{eq:rotation_decreases}. The last quantity is nonpositive by  \eqref{eq:per_tree_condition}.
\end{proof}

\begin{proof}[\textbf{Proof of Lemma~\ref{lem:azuma_hoeff_sup}}]
	The first two statements are immediate from the discussion preceding the Lemma. For the third statement, we have that
	\begin{align*}
	 P&\left(\sup_{f \in \F} \sum_{t=1}^n \epsilon_t f(\x_t) - 2B(f;\x_1,\ldots,\x_n) - \alpha \sum_{t=1}^n\sup_{f \in \F} f(\x_t)^2  >  u \right) \le P\left(\sum_{t=1}^n \pred_t \epsilon_t - \alpha \sum_{t=1}^n\sup_{f \in \F} f(\x_t)^2  >  u \right)
	 \end{align*}
	 and the latter probability is further upper bounded by
	 \begin{align*}
	 P\left(\sum_{t=1}^n \epsilon_t \pred_t  - \alpha \pred_t^2  >  u \right) &\le \inf_{\lambda > 0} \Es{\epsilon}{\exp\left\{ \lambda \sum_{t=1}^n \epsilon_t \pred_t  - \sum_{t=1}^n  \alpha \lambda \pred_t^2   - \lambda u\right\} } \\
	 & \le \inf_{\lambda > 0} \max_{\hat{y}_{1:n} \in [-1,1]^n}\exp\left\{ \sum_{t=1}^n \left( \lambda^2 \hat{y}_t^2/2  - \alpha \lambda \hat{y}_t^2\right)   - \lambda u\right\}  \le \exp\left\{ - 2\alpha u\right\}.
	\end{align*}
\end{proof}

\begin{lemma}
	\label{lem:balance_alpha}
 Suppose we have a collection of random variables $(X(g), Y(g))_{g\in\G}$, with $0\leq Y(g)\leq b$ almost surely for any $g\in\G$. Suppose for all $\alpha>0$, $c>0$, and some $0 \le a \le 1$, and $K\geq 0$ it holds that
 $$
 P\left(\sup_{g \in \G}\{ X(g) -  \alpha^{-a} K - \alpha Y(g)\} > u\right) \le \Gamma \exp\left\{ - c\alpha u\right\}~.
 $$
 Then 
 $$
 P\left( \sup_{g \in \G}\left\{ X(g) -   4K^\frac{1}{1+a} (Y(g)+1)^\frac{a}{a+1} - 4u\sqrt{Y(g) + 1}  \right\} > 0\right) \le  \log(b) \Gamma \exp\left\{ - c u^2 \right\}.
 $$ 
\end{lemma}
\begin{proof}[\textbf{Proof of Lemma~\ref{lem:balance_alpha}}]
Fix $u>0$ and consider a discretization over two regions $[d'_\ell, d'_u]$, $[d''_\ell, d''_u]$, given by $\alpha_i=  d'_\ell 2^{i-1}$, $i\in 1,\ldots, N'=\lceil\log (d'_u/d'_\ell) \rceil$ and $\alpha_j=  d''_\ell 2^{j-1}$, $j\in 1,\ldots, N''=\lceil\log (d''_u/d''_\ell) \rceil$. Let $N=N'+N''$ be the total cardinality of the discretization, and let $I$ denote the discretized set. From our premise, we have that for every index $i$ and $t>0$,
 $$
 P\left(\sup_{g \in \G}\{ X(g) -  \alpha_i^{-a} K - \alpha_i Y(g)\} > t\right) \le \exp\left\{ - c\alpha_i t \right\}~.
 $$
 Substituting $t=u \alpha_i^{-1} + \alpha_i$,
 $$
 P\left(\sup_{g \in \G}\{X(g)  -  \alpha_i^{-a} K - \alpha_i Y(g)\} - \alpha_i > u \alpha_i^{-1} \right) \le \exp\left\{ - cu - c\alpha_i^2 \right\}~.
 $$
 By union bound we conclude that,
\begin{align*}
P\left( \max_{\alpha_i\in I}\sup_{g \in \G} X(g)  -  \alpha_i^{-a} K  - \alpha_i Y(g) - \alpha_i - u \alpha_i^{-1} >0\right) \leq \sum_{\alpha_i\in I}   \exp\left\{ - cu -c\alpha_i^2\right\} \le N \times \exp\left\{- cu\right\}  ~.
\end{align*}
Therefore,
\begin{align*}
P\left( \sup_{g \in \G, \alpha}\{ X(g) - 2 \alpha (Y(g)+1)  -  \alpha^{-a} K  - u \alpha^{-1}\} >0\right) &  \le N \times \exp\left\{- cu\right\} 
\end{align*}
with $\alpha$ taking values in $[d'_\ell,d'_u]\cup[d''_\ell, d''_u]$. However,
\begin{align*}
	\inf_{\alpha} \left\{ 2 \alpha (Y(g)+1)  +  \alpha^{-a} K  + u \alpha^{-1} \right\} \leq \inf_{\alpha} \left\{ 2 \alpha (Y(g)+1)  +  2\max\{\alpha^{-a} K, u \alpha^{-1}\} \right\}.
\end{align*}
Passing to two balancing choices 
$$\alpha'=\sqrt{\frac{u}{Y(g)+1}},~~~~ \alpha'' = \left( \frac{K}{Y(g)+1}\right)^{1/(a+1)}$$ 
we obtain
$$
P\left( \sup_{g \in \G}\left\{ X(g) - 4 \max\left\{ \sqrt{(Y(g) + 1) u},  K^{1/(a+1)} (Y(g)+1)^{a/(a+1)}   \right\} \right\} > 0\right) \le  N \exp\left\{ - cu\right\} ~.
$$ 
It remains to quantify $N$ such that, for any $g\in\G$, the choices of $\alpha',\alpha''$ are captured by the two corresponding regions of discretization. It is immediate that $N$ does not depend on $u$ or $K$, and depends logarithmically on $b$. This concludes the proof.
\end{proof}

In the proofs, it is useful to work with an equivalent to \eqref{eq:def_r_growth} growth assumption \eqref{eq:assm_growth}, defined below.
\begin{lemma}
	\label{lem:assm_from_growth}
	Suppose sequential Rademacher complexity exhibits an $n^{1/r}$ growth with constant $D/2>0$, in the sense of \eqref{eq:def_r_growth}. Then the following holds for any $\{0,1\}$-valued predictable process $\bb$ and any $\X$-valued predictable process $\x$ (both with respect to the dyadic filtration):
	\begin{align}
		\label{eq:assm_growth}
		\Es{\epsilon}{ \sup_{f \in \F} \sum_{t=1}^n \epsilon_t \bb_t f(\x_t)} \le D \left(\max_{\epsilon}\sum_t \bb_t\right)^{1/r} \left(\sup_{f \in \F, t\in[n], \epsilon\in\{\pm1\}^n} |\bb_t f(\x_t)|\right).
	\end{align}
\end{lemma}
\begin{proof}[\textbf{Proof of Lemma~\ref{lem:assm_from_growth}}]
For any $f\in\F$,
\begin{align}
	\sum_{t=1}^n \epsilon_t \bb_t f(\x_t) = \sum_{i=1}^N \epsilon_{\tau_i} f(\x_{\tau_i})
\end{align}
where $N= \max_{\epsilon} \sum \bb_t$ and $\tau_i = \min\{s: \sum_{k=1}^s \bb_k \geq i\}$. For simplicity, assume $\sum \bb_t = N$ for all $\epsilon$ uniformly (the argument can be modified appropriately if not). Since $\bb_k$ is $\cA_{k-1}$-measurable, the event $\{\tau_i\leq t\}$ is $\cA_{t-1}$-measurable. Define $N$ random variables $\tilde{X}_i = \x_{\tau_i}$ and $\tilde{\epsilon}_i = \epsilon_{\tau_i}$, as well as the filtration $\tilde{\cA}_{i}=\cA_{\tau_i}$. We have that for any $f$ and $t$
$$\E\left[ \tilde{\epsilon}_i f(\tilde{X}_i)~\mid~ \tilde{\cA}_{i-1}\right] = 0$$
and therefore 
$\sum_{i=1}^N \tilde{\epsilon}_i f(\tilde{X}_i)$
is a sum of martingale differences, indexed by $f$. By the result of \cite{RakSriTew14}, for any process $\tilde{X}_1,\ldots,\tilde{X}_N$ with values in  $\text{img}(\x)$, 
$$\E\sup_{f\in\F} \sum_{i=1}^N \tilde{\epsilon}_i f(\tilde{X}_i) \leq 2 \sup_{\y,\x'} \E_\gamma \sup_{f\in\F} \sum_{i=1}^N \gamma_i \y_i f(\x'_i) = 2 \sup_{\x'} \E_\gamma \sup_{f\in\F} \sum_{i=1}^N \gamma_i f(\x'_i)$$
where $\y,\x'$ range, respectively, over $\{\pm1\}$-valued and $\text{img}(\x)$-valued trees. The last equality follows from the rotation lemma (see \cite{StatNotes2012}).
\end{proof}

\section{Proof of Lemma~\ref{lem:p_moment_lemma}}

\begin{lemma}
	\label{lem:uniform_in_eps}
	Let $\F\subseteq \reals^\X$ and $r\in(1,2]$. If \eqref{eq:def_r_growth} holds with constant $D/2$, then 
$$
 \Es{\epsilon}{ \sup_{f \in \F} \sum_{t=1}^n \epsilon_t f(\x_t)} \le C_{r,p}  \max_{\epsilon} \left( \sum_{t=1}^n \sup_{f \in \F} |f(\x_t)|^p \right)^{1/p}
$$
for any $1 \leq p <r$ and $C_{r,p} \deq D \left(1 - 2^{- (r - p)/rp}\right)^{-1}$.
\end{lemma}
\begin{proof}[\textbf{Proof of Lemma~\ref{lem:uniform_in_eps}}]
	
	Given a predictable process $\x$, define for each $k=0,1,\ldots,$ a predictable process $\bb^{(k)}$ by
	$$
	\bb^{(k)}_t = \left\{
	\begin{array}{ll}
	1 & ~\textrm{if }~~ 2^{-(k+1)/p} A  < \sup_{f \in \F} \left|f(\x_t)\right| \le 2^{-k/p} A\\
	0 & ~\textrm{otherwise} 
	\end{array}
	 \right.,
	$$
	where 
	$$A = \max_{\epsilon} \left( \sum_{t=1}^n \sup_{f \in \F} |f(\x_t)|^p \right)^{1/p}.$$
	Since $\x_t$ is $\cA_{t-1}$-measurable, so is $\bb^{(k)}_t$. From the definition, $\sum_{k\ge 0} \bb^{(k)}_t \equiv 1$. Hence
\begin{align*}
 \sup_{f \in \F} \sum_{t=1}^n \epsilon_t f(\x_t) \le \sum_{k\ge 0} \sup_{f \in \F} \sum_{t=1}^n \epsilon_t \bb_t^{(k)} f(\x_t).
\end{align*}
Denoting $N_k(\epsilon) = \{t : \bb^{(k)}_t =1 \}$,
\begin{align}
 \Es{\epsilon}{\sup_{f \in \F} \sum_{t=1}^n \epsilon_t f(\x_t)} & \le \sum_{k\ge 0}  \Es{\epsilon}{\sup_{f \in \F} \sum_{t=1}^n \epsilon_t \bb_t^{(k)} f(\x_t)} \notag\\ 
 & \le D \sum_{k\ge 0}  \left(\max_{\epsilon} \left|N_k(\epsilon)\right|\right)^{1/r}  \sup_{f,\epsilon , t} \left\{ |\bb_t^{(k)} f(\x_t)| \right\}\notag\\
 & \le D  A \sum_{k\ge 0}  \left(\max_{\epsilon} |N_k(\epsilon)|\right)^{1/r}   2^{-k/p}. \label{eq:nxt}
\end{align}
On the other hand note that for any $k=0,1,\ldots$, it holds that
\begin{align*}
 A &= \max_{\epsilon} \left( \sum_{t=1}^n \sup_{f \in \F} |f(\x_t)|^p \right)^{1/p}  \ge \max_{\epsilon} \left( \sum_{t=1}^n \sup_{f \in \F} |\bb_t^{(k)} f(\x_t)|^p \right)^{1/p} \\ 
 & \ge  \max_{\epsilon} \left( \left|N_k(\epsilon)\right|  \min_{t\in  N_k(\epsilon)}\sup_{f \in \F} |f(\x_t)|^p \right)^{1/p}   \\
 & \ge \left(\max_{\epsilon} \left|N_k(\epsilon)\right|\right)^{1/p}  2^{-(k+1)/p} A ~.
\end{align*}
Hence, $\max_{\epsilon} \left|N_k(\epsilon)\right| \le 2^{k+1}$. Using this in Eq. \eqref{eq:nxt} we conclude that:
\begin{align*}
 \Es{\epsilon}{\sup_{f \in \F} \sum_{t=1}^n \epsilon_t f(\x_t)} &  \le D  A \sum_{k\ge 0}  2^{(k+1)/r-k/p}  \le 2 D  A \sum_{k\ge 0}  2^{k/r-k/p} \le 2 D  A \sum_{k\ge 0}  2^{-k (r - p)/rp}\\
 & \le \frac{2 D  A}{1 - 2^{- (r - p)/rp}}  = \frac{2 D}{1 - 2^{- (r - p)/rp}}  \max_{\epsilon} \left( \sum_{t=1}^n \sup_{f \in \F} |f(\x_t)|^p \right)^{1/p} ~.
\end{align*}

\end{proof}

\begin{corollary}
	\label{cor:high_prob_unif_eps}
	Let $\F\subseteq \reals^\X$ and $r\in (1,2]$. If \eqref{eq:def_r_growth} holds with constant $D/4$, then 
\begin{align}
	\label{eq:high_prob_unif_eps}
\mathbb{P}\left( \left| \sup_{f \in \F} \sum_{t=1}^n \epsilon_t f(\x_t) \right| > 2C_{r,p}  \max_{\epsilon} \left( \sum_{t=1}^n \sup_{f \in \F} |f(\x_t)|^p \right)^{1/p} + u\right) \le \exp\left\{-\frac{u^2}{4 \max_{\epsilon} \sum_{t=1}^n\sup_{f \in \F} f(\x_t(\epsilon))^2} \right\}.
\end{align}
for any $1 \leq p <r$ and $C_{r,p} \deq D \left(1 - 2^{- (r - p)/rp}\right)^{-1}$.
\end{corollary}
\begin{proof}[\textbf{Proof of Corollary~\ref{cor:high_prob_unif_eps}}]
	First, if \eqref{eq:def_r_growth} holds for $\F$ with a constant $D/4$, then it holds for the class $\F\cup -\F$ with a constant $D/2$.
	It is immediate that 
	$$B(\x) = C_{r,p}  \max_{\epsilon} \left( \sum_{t=1}^n \sup_{f \in \F} |f(\x_t)|^p \right)^{1/p}$$
	satisfies \eqref{eq:rotation_decreases}. We now apply  Lemma~\ref{lem:uniform_in_eps} to the class $\F\cup-\F$, which yields that  \eqref{eq:per_tree_condition} of Lemma~\ref{lem:per_tree_abs_regret} is satisfied  for $\F\cup-\F$:
	\begin{align}
		\En  \left[ \sup_{f\in\F\cup-\F} \sum_{t=1}^n \epsilon_t f(\x_t) - 2 B(f; \x) \right] \leq 0.
	\end{align}	
The amplification argument of Lemma~\ref{lem:azuma_hoeff_sup} yields the tail bound for the supremum over $\F\cup-\F$, and the statement is concluded by noting that, pointwise,
$$\left| \sup_{f\in\F} \sum_{t=1}^n \epsilon_t f(\x_t)\right| \leq \sup_{f\in\F} \left|\sum_{t=1}^n \epsilon_t f(\x_t)\right| = \sup_{f\in\F\cup-\F} \sum_{t=1}^n \epsilon_t f(\x_t).$$
\end{proof}

\begin{proof}[\textbf{Proof of Lemma~\ref{lem:p_moment_lemma}, Part 2}]
	
	Given a predictable process $\x$ (with respect to the dyadic filtration), define  a predictable process $\bb$ by
	$$
	\bb_t = \left\{
	\begin{array}{ll}
	1 & \left( \sum_{s=1}^t \sup_{f \in \F} |f(\x_s)|^p \right)^{1/p} \leq a \\
	0 & ~\textrm{otherwise} 
	\end{array}
	 \right.,
	$$
	for some $a>0$, to be specified later. By definition, 
	\begin{align}
		\label{eq:control_a}
		\max_{\epsilon\in\{\pm1\}^n} \left( \sum_{t=1}^n \sup_{f \in \F} |\bb_t f(\x_t)|^p \right)^{1/p} \leq a .
	\end{align}
	Define the event
$
\mathcal{E} = \left\{ \epsilon : \left(\sum_{t=1}^{n} \sup_{f \in \F} |f(\x_t(\epsilon))|^p \right)^{1/p} > a \right\}
$
and note that for any $u>0$,
\begin{align}
\label{eq:truncation1}
 P\left(  \left|\sup_{f \in \F} \sum_{t=1}^n \epsilon_t f(\x_t) \right| > u\right) &\le P(\mathcal{E}) + P\left(\mathcal{E}^c,~ \left|\sup_{f \in \F} \sum_{t=1}^n \epsilon_t f(\x_t)\right| > u\right).
\end{align}
Consider the second term above:
\begin{align}
	\label{eq:truncation2}
	P\left(\mathcal{E}^c,~ \left|\sup_{f \in \F} \sum_{t=1}^n \epsilon_t f(\x_t)\right| > u\right) &= P\left(\mathcal{E}^c,~ \left|\sup_{f \in \F} \sum_{t=1}^n \epsilon_t \bb_t f(\x_t)\right| > u\right)\leq P\left(\left|\sup_{f \in \F} \sum_{t=1}^n \epsilon_t \bb_t f(\x_t)\right| > u\right).
\end{align}
Let $C_{r,p}= \frac{D}{1 - 2^{- (r - p)/rp}}$ and chose $a= \frac{u}{4C_{r,p}}$.  Given the choice of $a$, the last quantity can be written as
\begin{align}
	\label{eq:intermed1}
  P\left( \left| \sup_{f \in \F} \sum_{t=1}^{n} \epsilon_t \bb_t f(\x_t)\right| - \frac{2D}{1 - 2^{- (r - p)/rp}} a > u/2\right).
\end{align}
We would like to the tail bound of \eqref{eq:high_prob_unif_eps}. Observe that  functions in $\F$ only appear in \eqref{eq:high_prob_unif_eps} through their values on the predictable process $\x$. Hence, we may apply \eqref{eq:high_prob_unif_eps} to the collection $\{(\bb_t f(\x_t))_{t=1}^n: f\in\F\}$. In view of \eqref{eq:control_a}, the tail bound on \eqref{eq:intermed1} is
\begin{align}
	\label{eq:intermed2}
 \exp\left\{ -  \frac{u^2}{16 \max_{\epsilon} \sum_{t=1}^n \sup_{f \in \F} \bb_t f^2(\x_t(\epsilon))}\right\} .
\end{align}
Now note that for any $p \le 2$ and $\F\subseteq[-1,1]^\X$,
\begin{align*}
 \max_{\epsilon} \left(\sum_{t=1}^n \sup_{f \in \F} \bb_t f^2(\x_t)\right)^{1/2} &\le \min\left\{ \max_{\epsilon} \left(\sum_{t=1}^n \sup_{f \in \F} |\bb_t f(\x_t)|^p\right)^{1/p}, \max_{\epsilon} \left(\sum_{t=1}^n \sup_{f \in \F} f^2(\x_t)\right)^{1/2} \right\} \\
 & \leq \min\left\{a, \max_{\epsilon} \left(\sum_{t=1}^n \sup_{f \in \F} f^2(\x_t)\right)^{1/2}\right\}.
 \end{align*} 
Then the tail bound in \eqref{eq:intermed2} is
\begin{align}
	\label{eq:intermed3}
 \exp\left\{ -  \frac{u^2}{16 \min\{a^2, \max_{\epsilon} \left(\sum_{t=1}^n \sup_{f \in \F} f^2(\x_t)\right)\}}\right\} 
 &=  \min\left\{\exp\left\{ -  C_{r,p}^2 \right\} , \exp\left\{ -  \frac{u^2}{16 \max_{\epsilon} \left(\sum_{t=1}^n \sup_{f \in \F} f^2(\x_t)\right)}\right\} \right\} .
\end{align}
Combining with \eqref{eq:truncation1}, 
\begin{align*}
 P\left(  \left| \sup_{f \in \F} \sum_{t=1}^n \epsilon_t f(\x_t)\right| > u\right) &\le P\left(\left(\sum_{t=1}^{n} \sup_{f \in \F} |f(\x_t(\epsilon))|^p \right)^{1/p}  > \frac{u}{4 C_{r,p}}\right)  \\
 &+ \min\left\{\exp\left\{ -  C_{r,p}^2 \right\} , \exp\left\{ -  \frac{u^2}{16 \max_{\epsilon} \left(\sum_{t=1}^n \sup_{f \in \F} f^2(\x_t)\right)}\right\} \right\} .
\end{align*}
Now we have,
\begin{align}
	\label{eq:integrating_tail}
	\En \left|\sup_{f \in \F} \sum_{t=1}^n \epsilon_t f(\x_t)\right| & \le \int_{0}^{\infty} P\left(\sup_{f \in \F} \left|\sum_{t=1}^n \epsilon_t f(\x_t) \right|> u \right) du .
\end{align}
The integral is then controlled above by
\begin{align}
	\label{eq:intermed4}
	& 4 C_{r,p} \int_{0}^\infty P\left(\left(\sum_{t=1}^{n} \sup_{f \in \F} |f(\x_t(\epsilon))|^p \right)^{1/p}  > x\right)   dx + \int_{0}^\infty \min\left\{\exp\left\{ -  C_{r,p}^2 \right\} , \exp\left\{ -  \tfrac{u^2}{16 \max_{\epsilon} \left(\sum_{t=1}^n \sup_{f \in \F} f^2(\x_t)\right)}\right\} \right\}  du\\
  & = 4C_{r,p} \Es{\epsilon}{\left(\sum_{t=1}^{n} \sup_{f \in \F} |f(\x_t(\epsilon))|^p \right)^{1/p}} + \int_{0}^{4 C_{r,p}\max_{\epsilon} \left(\sum_{t=1}^n \sup_{f \in \F} f^2(\x_t)\right)^{1/2}} e^{ -  C_{r,p}^2 } du\\
  & ~~~~~~~~~ + \int_{4 C_{r,p}\max_{\epsilon} \left(\sum_{t=1}^n \sup_{f \in \F} f^2(\x_t)\right)^{1/2}}^\infty \exp\left\{ -  \frac{u^2}{16 \max_{\epsilon} \left(\sum_{t=1}^n \sup_{f \in \F} f^2(\x_t)\right)}\right\}  du \notag\\
  &\leq 4C_{r,p}   \Es{\epsilon}{\left(\sum_{t=1}^{n}\sup_{f \in \F} |f(\x_t)|^p \right)^{1/p}}   + 8 C_{r,p}\max_{\epsilon} \left(\sum_{t=1}^n \sup_{f \in \F} f^2(\x_t)\right)^{1/2} \exp\left\{ -  C_{r,p}^2 \right\}.
\end{align}
Trivially, $x/2 \leq 2^x-1 \leq x$ for $x\in[0,1]$. Hence, we can upper and lower bound $C_{r,p}$ as
\begin{align}
	\frac{D r p 2^{\frac{r-p}{rp}}}{ r - p} \leq C_{r,p}  \leq \frac{2 D r p 2^{\frac{r-p}{rp}}}{r-p} ~.
\end{align}
Since $1 < p < r \le 2$, we have
\begin{align}
	\frac{D}{ r - p} \leq C_{r,p}  \leq \frac{8 D}{r-p} ~.
\end{align}
Hence, we conclude that
\begin{align*}
	\En \left|\sup_{f \in \F} \sum_{t=1}^n \epsilon_t f(\x_t)\right| & \leq \frac{32 D}{r-p}   \Es{\epsilon}{\left(\sum_{t=1}^{n}\sup_{f \in \F} |f(\x_t)|^p \right)^{1/p}}   + \frac{64D}{r-p} \max_{\epsilon} \left(\sum_{t=1}^n \sup_{f \in \F} f^2(\x_t)\right)^{1/2} \exp\left\{ -  \frac{D^2}{(r-p)^2} \right\}\\
	& \le \frac{32 D}{r-p}   \Es{\epsilon}{\left(\sum_{t=1}^{n}\sup_{f \in \F} |f(\x_t)|^p \right)^{1/p}}   + \frac{64D}{r-p} \sqrt{n} \exp\left\{ -  \frac{D^2}{(r-p)^2} \right\}
\end{align*}
Now say we set $\frac{1}{p} = \frac{1}{r} + \frac{1}{\log n}$, then in this case, $\frac{4}{\log n} \ge r - p >\frac{1}{\log n}$ and,
\begin{align}
	\Es{\epsilon}{\left(\sum_{t=1}^{n}\sup_{f \in \F} |f(\x_t)|^p \right)^{1/p}} \leq n^{\frac{1}{p}-\frac{1}{r}}\Es{\epsilon}{\left(\sum_{t=1}^{n}\sup_{f \in \F} |f(\x_t)|^r \right)^{1/r}} \leq 2 \Es{\epsilon}{\left(\sum_{t=1}^{n}\sup_{f \in \F} |f(\x_t)|^r \right)^{1/r}}.
\end{align}
We conclude that
\begin{align*}
	\En \left|\sup_{f \in \F} \sum_{t=1}^n \epsilon_t f(\x_t)\right| & \leq  32 D\ \log n\   \Es{\epsilon}{\left(\sum_{t=1}^{n}\sup_{f \in \F} |f(\x_t)|^r \right)^{1/r}}   + \frac{64 D \sqrt{n} \log n}{n^{D^2 \log n}}~.
\end{align*}
\end{proof}

\begin{lemma}
	\label{lem:weak_control}
 	Let $\F\subseteq \reals^\X$ and $r\in(1,2]$. Under the growth assumption \eqref{eq:def_r_growth} with constant $D/2$, for any $p < r$, $u>0$, and $C_{r,p} \deq D \left(1 - 2^{- (r - p)/rp}\right)^{-1}$, 
\begin{align}
 P\left( \max_{s\leq n} \left| \sup_{f \in \F} \sum_{t=1}^s \epsilon_t f(\x_t)\right| > u\right)  \le 2 \left(\frac{C_{r,p}^{p} \Es{}{\sum_{t=1}^{n} \sup_{f \in \F} |f(\x_t)|^{p}} }{u^{p}}\right)^{\frac{1}{p+1}}. 
\end{align}
\end{lemma}
\begin{proof}[\textbf{Proof of Lemma~\ref{lem:weak_control}}]
As in the proof of Lemma~\ref{lem:p_moment_lemma}, given a predictable process $\x$, define  a predictable process $\bb$ by
	$
	\bb_t = \ind{ \left( \sum_{s=1}^t \sup_{f \in \F} |f(\x_s)|^{p} \right)^{1/p} \leq a}$ 	for some $a>0$, to be specified later. 
	Define 
$
\mathcal{E} = \left\{ \epsilon : \left(\sum_{t=1}^{n} \sup_{f \in \F} |f(\x_t(\epsilon))|^{p} \right)^{1/p} > a \right\}.
$
For any $u>0$,
\begin{align}
	\label{eq:intermed1c}
 P\left(  \max_{s\leq n} \left| \sup_{f \in \F} \sum_{t=1}^s \epsilon_t f(\x_t)\right| > u\right) &\le P(\mathcal{E}) + P\left(\mathcal{E}^c,~ \max_{s\leq n}\left| \sup_{f \in \F} \sum_{t=1}^s \epsilon_t f(\x_t)\right| > u\right).
\end{align}
The second term can be written as  
\begin{align*}
		P\left(\mathcal{E}^c,~ \max_{s\leq n} \left| \sup_{f \in \F} \sum_{t=1}^s \epsilon_t f(\x_t) \right| > u \right) &= P\left(\mathcal{E}^c,~ \max_{s\leq n}\left| \sup_{f \in \F} \sum_{t=1}^s \epsilon_t \bb_t f(\x_t) \right| > u\right)\leq P\left(\max_{s\leq n}\left| \sup_{f \in \F} \sum_{t=1}^n \epsilon_t \bb_t f(\x_t)\right| > u\right)
\end{align*}
and further upper bounded, using Lemma~\ref{lem:uniform_in_eps}, by
\begin{align}
	\label{eq:intermed1d}
  P\left( \max_{s\leq n}\left| \sup_{f \in \F} \sum_{t=1}^{s} \epsilon_t \bb_t f(\x_t) \right| > u\right) \le \frac{\Es{}{\left|\sup_{f \in \F} \sum_{t=1}^{n} \epsilon_t \bb_t f(\x_t) \right|}}{u} \le \frac{C_{r,p} a}{u}.
\end{align}
The first inequality is implied by Doob's maximal inequality since 
$s\mapsto \sup_{f \in \F} \sum_{t=1}^{s} \epsilon_t \bb_t f(\x_t(\epsilon))$ is a discrete-time submartingale. Next,
$$P(\mathcal{E}) = P\left(\sum_{t=1}^{n} \sup_{f \in \F} |f(\x_t)|^{p}  > a^p\right) \leq \frac{\Es{}{\sum_{t=1}^{n} \sup_{f \in \F} |f(\x_t)|^{p}}}{a^p}.$$
Combining with \eqref{eq:intermed1c}, we have
\begin{align}
 P\left( \max_{s\leq n} \left| \sup_{f \in \F} \sum_{t=1}^s \epsilon_t f(\x_t) \right| > u\right) &\le  \frac{\Es{}{\sum_{t=1}^{n} \sup_{f \in \F} |f(\x_t)|^{p}}}{a^p} + \frac{C_{r,p} a}{u}. 
\end{align}
Setting $a = \left(\frac{u \left(\Es{}{\sum_{t=1}^{n} \sup_{f \in \F} |f(\x_t)|^{p}}\right)}{C_{r,p}}\right)^{1/(p+1)}$, yields the statement of the lemma.
\end{proof}

\newcommand{\btildex}{\tilde{{\mathbf{x}}}}

\begin{lemma}
	\label{lem:pmomentbound}
 	Let $\F\subseteq \reals^\X$ and $r\in(1,2]$. Under the growth assumption \eqref{eq:def_r_growth}, for any $p < r$ there exists a constant $B_{r,p}<\infty$ such that for any $u>0$,
	\begin{align}
	 P\left(  \left| \sup_{f \in \F} \sum_{t=1}^n \epsilon_t f(\x_t)\right| > u\right)  \le \frac{B_{r,p} \En\left[ \sum_{t=1}^{n} \sup_{f \in \F} |f(\x_t)|^{p} \right]}{u^p}.  
	\end{align}
\end{lemma}
\begin{proof}[\textbf{Proof of Lemma~\ref{lem:pmomentbound}}]
	Given a predictable $\X$-valued process $\x=(\x_1,\ldots,\x_n)$ of length $n$, consider a predictable process $\btildex$ of length $n N$ constructed by concatenating $N$ copies of $\x$. Let $\epsilon^1, \ldots, \epsilon^N \in\{\pm1\}^n$ be $N$ independent vectors with i.i.d. Rademacher coordinates, and let $\epsilon=(\epsilon^1,\ldots,\epsilon^N)$. Denote $Z(\epsilon^j) \deq  \left| \sup_{f \in \F} \sum_{t=1}^n \epsilon^j_t f(\x_t(\epsilon^j)) \right|$.   
Now note that
\begin{align*}
\max_{j \leq N} Z(\epsilon^{j}) &= \max_{j \leq N} \left| \sup_{f \in \F} \left( \sum_{t=1}^{n j} \epsilon_t f(\btildex_t(\epsilon))-  \sum_{t=1}^{n (j-1)} \epsilon_t f(\btildex_t(\epsilon))  \right) \right| \\
&\le\max_{j \leq N}\left( \left|\sup_{f \in \F}\sum_{t=1}^{n j} \epsilon_t f(\btildex_t) \right| +  \left| \sup_{f \in \F} \sum_{t=1}^{n (j-1)} -\epsilon_t f(\btildex_t) \right| \right)\\
&\leq \max_{s \leq nN} \left| \sup_{f \in \F} \sum_{t=1}^{s} \epsilon_{t} f(\tilde{\x}_{t})\right| + \max_{s \leq nN} \left| \sup_{f \in \F} \sum_{t=1}^{s} -\epsilon_{t} f(\tilde{\x}_{t}) \right| ~.
\end{align*}
The two terms in the last bound have the same distribution, and so
\begin{align*}
P\left( N^{-1/p} \max_{j \in [N]} Z(\epsilon^{j}) > u \right) 
& \le 2 P\left( \max_{k \in [nN]} \left| \sup_{f \in \F} \sum_{t=1}^{k} \epsilon_{t}  f(\tilde{\x}_{t}) \right| > u N^{1/p}/2 \right)\\
& \le 4 \left(\frac{C_{r,p}^{p} \Es{\epsilon}{\sum_{t=1}^{n N} \sup_{f \in \F} |f(\tilde{\x}_t(\epsilon))|^{p}} }{N u^{p}}\right)^{\frac{1}{p+1}}.
\end{align*}
Since $\btildex$ is a concatenation of $N$ copies of $\x$, convexity of $\sup$ implies
$$
\Es{\epsilon}{ \sum_{t=1}^{n N} \sup_{f \in \F}  |f(\tilde{\x}_t)|^p } \leq N\ \Es{\epsilon}{  \sum_{t=1}^n \sup_{f \in \F} |f(\x_t)|^p }. 
$$
By homogeneity of the lemma statement, we may assume 
$\Es{\epsilon}{  \sum_{t=1}^n \sup_{f \in \F} |f(\x_t)|^p } = 1.$
With this scaling, we have proved
\begin{align*}
\sup_{N\geq 1} P\left( N^{-1/p} \max_{j \leq N} Z(\epsilon^{j}) > u \right) & \le 4 \left(\frac{C_{r,p}^{p} }{ u^{p}}\right)^{\frac{1}{p+1}}.
\end{align*}
Define $u_{r,p}$ to be the value of $u$ that makes the right-hand side equal to $1/2$. 
The reverse H\"older principle (see Proposition 8.53 in \cite{Pisier11}, originating in the work of D.L. Burkholder \cite{burkholder1964maximal}) implies that there exists constant $B_{r,p}<\infty$ such that
$$
P\left( \left| \sup_{f \in \F} \sum_{t=1}^n \epsilon_t f(\x_t)\right| > u\right)  \le B_{r,p} u^{-p}. 
$$
The statement follows by homogeneity.
\end{proof}

\begin{proof}[\textbf{Proof of Lemma~\ref{lem:p_moment_lemma}, Part 1}]
Define the predictable process $\bb$ as 
	$
	\bb_t = \ind{ \sum_{s=1}^t \sup_{f \in \F} |f(\x_s)|^{p}  \leq u^p }
	$
 and the event
$
\mathcal{E} = \left\{ \epsilon : \left(\sum_{t=1}^{n} \sup_{f \in \F} |f(\x_t(\epsilon))|^{p} \right)^{1/p} > u\right\}
$. 
As in the proof of the other part of the lemma,
\begin{align*}
 P\left( \left| \sup_{f \in \F} \sum_{t=1}^n \epsilon_t f(\x_t) \right| > u \right) \le P(\mathcal{E}) +  P\left(  \left| \sup_{f \in \F} \sum_{t=1}^n \epsilon_t \bb_t f(\x_t) \right| > u\right) ~. 
\end{align*}
Under the growth assumption \eqref{eq:def_r_growth}, applying Lemma~\ref{lem:pmomentbound}, there exists a finite constant $B_{r,p}$ such that
\begin{align*}
 P\left( \left| \sup_{f \in \F} \sum_{t=1}^n \epsilon_t \bb_t f(\x_t) \right| > u\right) &\le \frac{B_{r,p} \Es{}{\sum_{t=1}^{n} \sup_{f \in \F} \bb_t |f(\x_t)|^{p}}}{u^{p}} \\
 & \le \frac{B_{r,p} \Es{}{\min\left\{\sum_{t=1}^{n} \sup_{f \in \F}  |f(\x_t)|^{p}, u^{p}\right\}}}{u^{p}}. 
 \end{align*}
 The last inequality follows from the definition of $\bb_t$.
We then have
\begin{align*}
 \Es{}{\left| \sup_{f \in \F} \sum_{t=1}^n \epsilon_t f(\x_t)  \right|} &=  \int_0^\infty   P\left( \sup_{f \in \F} \sum_{t=1}^n \epsilon_t f(\x_t) > u \right) du\\
 &\le  \int_0^\infty  P\left( \mathcal{E} \right) du + B_{r,p} \int_0^\infty  u^{-p} \Es{\epsilon}{\min\left(\sum_{t=1}^{n}\sup_{f \in \F}  |f(\x_t)|^{p} , u^{p}\right)} du. 
 \end{align*}
 Let us focus on the second integral. Exchanging the order of integration and splitting the integral into two parts, 
 \begin{align*}
 	&\int_0^\infty  u^{-p} \Es{\epsilon}{\min\left(\sum_{t=1}^{n}\sup_{f \in \F}  |f(\x_t)|^{p} , u^{p}\right)} du \\
  	&=\Es{\epsilon}{\int_{0}^{\left(\sum_{t=1}^{n}\sup_{f \in \F}  |f(\x_t)|^{p}\right)^{1/p}}  1 du} + \Es{\epsilon}{\int_{\left(\sum_{t=1}^{n}\sup_{f \in \F}  |f(\x_t)|^{p}\right)^{1/p}}^\infty \left(\sum_{t=1}^{n}\sup_{f \in \F}  |f(\x_t)|^{p}\right) u^{-p}   du }  \\
 &= \Es{\epsilon}{\left(\sum_{t=1}^{n}\sup_{f \in \F}  |f(\x_t)|^{p}\right)^{\frac{1}{p}}}  + \frac{1}{p-1} \Es{\epsilon}{\left(\sum_{t=1}^{n}\sup_{f \in \F}  |f(\x_t)|^{p}\right)   \left(\sum_{t=1}^{n}\sup_{f \in \F}  |f(\x_t)|^{p}\right)^{(1-p)/p} }  \\
  &= \left(1 +  \frac{1}{p-1}\right) \Es{\epsilon}{\left(\sum_{t=1}^{n}\sup_{f \in \F}  |f(\x_t)|^{p}\right)^{1/p} } .
\end{align*}
By the definition of $\mathcal{E}$ we also have
$$
\int_{0}^\infty  P\left( \mathcal{E} \right) du  = \Es{}{\left(\sum_{t=1}^{n} \sup_{f \in \F} |f(\x_t)|^{p}\right)^{1/p} } .
$$
Combining,
\begin{align*}
 \Es{}{\left| \sup_{f \in \F} \sum_{t=1}^n \epsilon_t f(\x_t)  \right|} &\le \left( 1 + B_{r,p}\left(1 +  \frac{1}{p-1}\right)\right) \En \left[ \left(\sum_{t=1}^{n}\sup_{f \in \F}  |f(\x_t)|^{p}\right)^{1/p} \right].
\end{align*}

\end{proof}

\section{Proofs of Theorem~\ref{thm:unif} and Theorem~\ref{thm:perfunction} }

\begin{proof}[\textbf{Proof of Theorem~\ref{thm:unif}}]
	We first prove the lemma for a class $\G\subseteq [-1/2,1/2]^\cZ$ and then rescale the final result, using the homogeneity of the bound. For the derived class $\F=\{f_g(z,z')=g(z)-g(z'): g\in\G\}\subseteq [-1,1]^\cX$, where $\X=\cZ\times\cZ$, \eqref{eq:def_r_growth} holds with a constant $D/2$. Invoking \eqref{eq:exp_tail_sup} of Lemma~\ref{lem:azuma_hoeff_sup}, for any $\X$-valued predictable process $\x$,
\begin{align*}
 & P\left(\sup_{f \in \F} \sum_{t=1}^n \epsilon_t f(\x_t) - B(f; \x_1,\ldots,\x_n) >  u \right) \le \exp\left( - 2\alpha u\right)
\end{align*}
with 
\begin{align*}
	B(f; \x_1,\ldots,\x_n) = 32 D \log_2(n)  \left(\sum_{t=1}^{n}\sup_{f \in \F} |f(\x_t)|^r \right)^{1/r}  - \alpha \sum_{t=1}^n\sup_{f \in \F} f(\x_t)^2 - \phi_n
\end{align*} 
and $\phi_n = \frac{64 D \sqrt{n} \log n}{n^{D^2 \log n}}$.
	Define 
	$$\tilde{B}(g; z_1,z'_1,\ldots,z_n,z'_n) = 32 D \log_2(n)  \left(\sum_{t=1}^{n}\sup_{g \in \G} |g(z_t)-g(z'_t)|^r \right)^{1/r}  - \alpha \sum_{t=1}^n\sup_{g \in \G} (g(z_t)-g(z'_t))^2 - \phi_n.$$
	Condition \eqref{eq:symmetry_condition_B} of Corollary~\ref{cor:prob_symmetrization} is clearly satisfied for this function. Hence,
	\begin{align*}
		&P\left(\sup_{g\in\G} \sum_{t=1}^n (g(Z_t)-\En_{t-1} g(Z_t)) - \En_{Z_{1:n}'} \tilde{B}(g; Z_1,Z_1',\ldots,Z_n, Z_n')>u\right) \leq  \exp\{1- 2\alpha u\}
	\end{align*}
	We now apply Lemma~\ref{lem:balance_alpha} with $K=0$, $a=0$, 
	$$X(g) = \sum_{t=1}^n (g(Z_t)-\En_{t-1} g(Z_t)) - 32 D \log_2(n)  \En_{Z_{1:n}'}\left(\sum_{t=1}^{n}\sup_{g \in \G} |g(Z_t)-g(Z'_t)|^r \right)^{1/r}- \phi_n$$ 
	$$Y(g) = \En_{Z_{1:n}'} \sum_{t=1}^n\sup_{g \in \G} (g(Z_t)-g(Z'_t))^2$$
	yielding
	\begin{align*}
		&P\left( \sup_{g\in\G} \sum_{t=1}^n (g(Z_t)-\En_{t-1} g(Z_t)) - 128 D \log_2(n)  \En_{Z_{1:n}'}\left(\sum_{t=1}^{n}\sup_{g \in \G} |g(Z_t)-g(Z'_t)|^r \right)^{1/r} - 4\phi_n \right.\\
		&\hspace{2in}\left. - 4u\sqrt{\En_{Z_{1:n}'}\sum_{t=1}^n\sup_{g \in \G} (g(Z_t)-g(Z'_t))^2+1} > 0 \right) \leq \log(n)\exp(1-2u^2).
	\end{align*}
	Using Jensen's inequality, we push the expectation inside the $1/r$ power, completing the proof. 
\end{proof}

\begin{proof}[\textbf{Proof of Theorem~\ref{thm:perfunction}}]
	Define 
	$$\tilde{B}(g; z_1,z'_1,\ldots,z_n,z'_n) = \frac{\alpha}{2}\sum_{t=1}^{n} (g(z_t)-g(z'_t))^2  + C \alpha^{-\frac{2-q}{2+q}} n^{\frac{q}{2+q}}.$$
	Condition \eqref{eq:symmetry_condition_B} of Corollary~\ref{cor:prob_symmetrization} is  satisfied and, therefore, 
	\begin{align*}
		&P\left(\sup_{g\in\G} \sum_{t=1}^n (g(Z_t)-\En_{t-1} g(Z_t)) - \En_{Z_{1:n}'} \tilde{B}(g; Z_1,Z_1',\ldots,Z_n, Z_n')>u\right) \leq  \exp\{1- \alpha u\}.
	\end{align*}
	To invoke Lemma~\ref{lem:balance_alpha}, we choose $a= \frac{2-q}{2+q}$, $K = Cn^{\frac{q}{2+q}},$
	$$X(g)=\sum_{t=1}^n (g(Z_t)-\En_{t-1} g(Z_t)),~~~  Y(g)=\frac{1}{2}\En_{Z_{1:n}'}\sum_{t=1}^{n} (g(Z_t)-g(Z'_t))^2.$$
	Writing $\frac{1}{1+a}=\frac{2+q}{4}$, we conclude that with probability at least $1-e\log (n)\exp\{-u^2\}$
	\begin{align}
		\sup_{g\in\G} \sum_{t=1}^n (g(Z_t)-\En_{t-1} g(Z_t)) &\leq  C'n^{\frac{q}{4}}\left(\frac{1}{2}\En_{Z_{1:n}'}\sum_{t=1}^{n} (g(Z_t)-g(Z'_t))^2 +1\right)^\frac{2-q}{4}  \\
		&+ 4u\sqrt{\frac{1}{2}\En_{Z_{1:n}'}\sum_{t=1}^{n} (g(Z_t)-g(Z'_t))^2 +1} ~.
	\end{align}
\end{proof}